\documentclass[10pt]{article}
\usepackage{mathrsfs}
\usepackage[numbers,sort&compress]{natbib}
\usepackage{amsthm,amsmath,amssymb,anysize}
\newtheorem{lemma}{Lemma}[section]
\newtheorem{theorem}[lemma]{Theorem}

\newtheorem{proposition}[lemma]{Proposition}
\newtheorem{definition}[lemma]{Definition}
\newtheorem{corollary}[lemma]{Corollary}
\newtheorem{example}{Example}[section]
\marginsize{3.6cm}{3.6cm}{3.6cm}{3cm}
\numberwithin{equation}{section}

\title{{\textbf{Fine gradings and their Weyl groups for twisted Heisenberg Lie superalgebras}}\footnote{Supported by the NNSF of China (11471090, 11071187, 11171055)}}
\author{\textbf{Wenjuan Xie and Wende Liu\footnote{Corresponding author. E-mail address: wendeliu@ustc.edu.cn.}}\\
 \textit{School of Mathematical Sciences,} \\\textit{Harbin Normal University, Harbin 150025, China}\\
 }

\date{ }
\begin{document}

\maketitle
\begin{quotation}
{\noindent\textbf{Abstract}   In this  paper  we  define the so-called  twisted Heisenberg superalgebras over the complex number field by adding derivations to Heisenberg superalgebras. We  classify the fine  gradings up to equivalence on  twisted Heisenberg superalgebras and determine the Weyl groups of those gradings.}
 \\

{\noindent\textbf{Keywords}\quad Heisenberg superalgebra, fine grading, Weyl group}

 {\noindent
\textbf{2010 MR Subject Classification}  17B70, 17B40}
  \end{quotation}
\setcounter{section}{0}

\section{Introduction} \label{sect1}
 Group gradings (or simply, gradings) on Lie algebras emerged   at the beginning of Lie Theory and since then have been extensively used. The Cartan decomposition of a semisimple Lie algebra is one of the most interesting
fine gradings, which has a heavy influence in Lie algebras and their representation theory. There are also numerous applications of gradings on various algebraic structures. For instance, gradings on the tangent Lie algebras of  Lie groups are  used in symmetric spaces; finite cyclic group gradings of finite dimensional semisimple Lie algebras are used in Kac-Moody algebra theory, etc.

In the last ten years there has been an increasing interest in the  gradings on a well-known algebraic structure   and the study used to be focused on the simple ones including simple Lie algebras and simple Lie superalgebras. For instance, a classification
 of fine gradings up to equivalence on all classical simple Lie algebras  over an algebraically closed field of characteristic 0 has
been completed
(see \cite{Eld1}). For further information  pertaining to gradings on  various simple algebras  the reader is referred to \cite{Bah1,Bah2,Bah3,Bah4,Cal1,EMG,DM,DMV,DV,Eld0,Eld2,Eld3,EK,Ko} and the references therein. Recently, the gradings on  non-simple Lie algebras have also been considered.  Among them the gradings by abelian groups on  filiform algebras (which are nilpotent Lie algebras)  are
classified up to isomorphism (see \cite{Bah5}) and  the fine  gradings on  Heisenberg (super)algebras and twisted Heisenberg algebras (which are nilpotent or solvable Lie superalgebras) are classified up to equivalence (see \cite{Cal2}).

One of the important ingredients in the group grading theory is the  notion of Weyl groups, which is a generalization of the usual one for Lie algebras and in which the symmetry for a graded algebra is included  in a sense.
 Patera and Zassenhaus initiated the study of Weyl groups of Lie gradings (see \cite{Patera}).   The reader is also referred to more recent papers on Weyl groups of gradings  (see \cite{Eld4,Eld5}).

We are concerned with the so-called twisted Heisenberg superalgebras, which are  an infinite family of finite dimensional solvable Lie superalgebras, obtained by forming a semi-direct product of  a Heisenberg Lie superalgebra  and a derivation. It is well known that, as  in the non-super  case, Heisenberg superalgebras and their twisted deformations are very important
subjects in the research of both mathematics and physics.  In this paper, starting from defining twisted Heisenberg  superalgebras, we first classify all their fine gradings up to equivalence and
characterize their Weyl groups in a simplest case.  Compared with twisted Heisenberg algebras, we can find more than two fine gradings up to equivalence  and their Weyl groups are clear and concise in this case.
Finally, in terms of
blocks, we classify all the fine gradings  up to equivalence  on complex twisted Heisenberg  superalgebras. Furthermore, we determine all the Weyl groups of these gradings.

\section{Preliminaries} \label{ns}
All algebras in this paper are assumed to be over
the complex number field $\mathbb{C}$ and of finite dimensions. We should note that the main results also hold over an algebraically closed field of characteristic 0.

 We   briefly review some basic definitions  pertaining to gradings of superalgebras and for more details in non-super case the reader is referred to \cite{Cal2,Bah5,Eld1,Eld3}.
Let $A$ be a  superalgebra and   $G$  an abelian group.  A $G$-\textit{grading}  on $A$ is a $\mathbb{Z}_{2}$-graded vector space decomposition by $G$,
$$\Gamma:A=\bigoplus_{g\in G}A_{g},$$
 such that $A_{g}A_{h}\subset A_{gh}$ for all $g,h\in G$, in which each $A_{g}$ is called a homogeneous  component of degree $g$. The \textit{support} of the $G$-grading is the set $$S=\{g\in G\mid A_{g}\neq 0\}.$$
Suppose we are given another grading on the superalgebra $A_{g}$,
$$\Gamma': A=\bigoplus_{h\in H}A'_{h}$$
with support $S'$, where $H$ is an abelian group. The gradings $\Gamma$ and $\Gamma'$ are  said to be \textit{equivalent} if there is  a  superalgebra automorphism $\phi$ of $A$ equipped with a bijection $\sigma: S\longrightarrow S'$ such that $\phi(A_{s}) =A'_{\sigma(s)}$ for all $s\in S$; such an automorphism $\phi$ is called an equivalence between the gradings $\Gamma$ and $\Gamma'$.

The \textit{automorphism group} of $\Gamma$, denoted by $\mathrm{Aut}(\Gamma)$, is the group consisting of all the self-equivalences of $\Gamma$, i.e., automorphisms of the superalgebra $A$ which permute the components of $\Gamma$. The \textit{stabilizer} of $\Gamma$, denoted by $\mathrm{Stab}(\Gamma)$, is the subgroup of $\mathrm{Aut}(\Gamma)$ consisting of all the automorphisms of the superalgebra $A$ that leave each component of $\Gamma$ invariant. The quotient group $\mathrm{Aut}(\Gamma)/ \mathrm{Stab}(\Gamma)$ is called the \textit{Weyl group} of $\Gamma$, which is denoted by $\mathcal{W}(\Gamma)$. The \textit{diagonal group} of $\Gamma$, denoted by $\mathrm{Diag}(\Gamma)$,  consists of all automorphisms of the graded superalgebra $A$ such that each component $A_{g}$ of $\Gamma$ contained in some eigenspace of every $f\in \mathrm{Diag}(\Gamma)$. The grading $\Gamma$ of a superalgebra $A$ is said to be \textit{toral} if $\mathrm{Diag}(\Gamma)$ is contained in a torus of the algebraic group $\mathrm{Aut}(A)$.

A grading $\Gamma$ is called a \textit{refinement} of a grading $\Gamma'$ (or $\Gamma'$ a \textit{coarsening} of $\Gamma$) if each homogeneous component of $\Gamma'$ is a (direct) sum of some homogeneous components of $\Gamma$. A grading is said to be \textit{fine} if it admits no proper refinements in the obvious sense.

A $G$-grading $\Gamma:A=\bigoplus_{g\in G}A_{g}$ is said to be \textit{universal} if the following universal property holds: for any coarsening $A=\bigoplus_{g'\in G'}A'_{g'}$ of $\Gamma$ there exists a unique group epimorphism $\alpha:
 G\longrightarrow G'$ such that for $g'\in G'$,
 $$A'_{g'}=\bigoplus\limits_{g\in   \alpha^{-1}(g')}A_{g}.$$

  For any $G$-grading $\Gamma: A=\bigoplus_{g\in G}A_{g}$, there is a universal grading   equivalent to $\Gamma$. In fact,  if we write  $\widetilde{G}$ for the abelian group generated by the support of $\Gamma$ subject to the relations $g_{1}g_{2}=g_{3}$ if $0\neq[A_{g_{1}}, A_{g_{2}}] \subset A_{g_{3}}$, then  $\widetilde{\Gamma}: A=\bigoplus_{\widetilde{g}\in \widetilde{G}}A_{\widetilde{g}}$ is a universal  grading   equivalent to  $\Gamma$, where $A_{\widetilde{g}}$ is the sum of all the homogeneous components $A_{g}$ of $\Gamma$ such that the class of  $g$ in $\widetilde{G}$ is $\widetilde{g}$ (see \cite[Section 3.3]{Ko}).  The group $\widetilde{G}$ is called the \textit{universal grading group} of $\Gamma$. Recall that a fine grading is toral if and only if its universal grading group is torsion-free (see \cite[Section 2]{Cal2}).
  Since $A$ is a finite dimensional algebra,   $\widetilde{G}$ is a finitely generated abelian group.
In view of this fact, in the subsequent sections all the gradings are assumed to be  universal.

\section{Twisted Heisenberg superalgebras}

  Recall that  a Heisenberg superalgebra  is  a two-step nilpotent Lie superalgebra  with one-dimensional
even center (see \cite{Cal2}).   The even part of a Heisenberg superalgebra  is a Heisenberg algebra and hence must be of odd dimension. A Heisenberg superalgebra  of super-dimension $(n, m)$ with $n=2k+1$ has a homogeneous basis
\begin{equation}\label{heisenberg basis} \{z,e_{1},\widehat{e_{1}},\ldots,e_{k},\widehat{e_{k}}\mid w_{1},\ldots,w_{m}\}
\end{equation}
with nonzero brackets
$$[e_{i},\widehat{e_{i}}]=-[\widehat{e_{i}},e_{i}]=[w_{j},w_{j}]=z,$$
where $1\leq i\leq k, 1\leq j\leq m$. We denote by $H_{n,m}$ the Heisenberg superalgebra  of super-dimension $(n, m)$.

Let us define twisted Heisenberg superalgebra.
\begin{definition}
Let $n$ be a positive odd number and $m\in \mathbb{Z}_{\geq 0}$. Write $n=2k+1$ and suppose $2r\leq m,$ where $k,r\in \mathbb{Z}_{\geq 0}$. For $\lambda=(\lambda_{1},\ldots,\lambda_{k})\in (\mathbb{C^{\times}})^{k}$ and $\kappa=(\kappa_{1},\ldots,\kappa_{r})\in (\mathbb{C^{\times}})^{r}$,  a superspace of superdimension $(n+1, m)$
 with a homogeneous basis
\begin{equation}\label{first basis}
\{z,u,e_{1},\widehat{e_{1}},\ldots,e_{k},\widehat{e_{k}}\mid w_{1},\widehat{w_{1}},\ldots,w_{r},\widehat{w_{r}},\eta_{1},\ldots,\eta_{m-2r}\},
\end{equation}
becomes a Lie superalgebra with respect to the multiplication given by the following  nonzero   brackets ($i\leq k, j\leq r$, $t\leq m-2r$)
\begin{equation*}\begin{split}
&[e_{i},\widehat{e_{i}}]=\lambda_{i}z,~~[u,e_{i}]=\lambda_{i}\widehat{e_{i}},~~[u,\widehat{e_{i}}]=-\lambda_{i}e_{i},\\
&
 [u,w_{j}]=\kappa_{j}\widehat{w_{j}},~~[u,\widehat{w_{j}}]=-\kappa_{j}w_{j},\\
&[w_{j},w_{j}]=[\widehat{w_{j}},\widehat{w_{j}}]=[\eta_{t},\eta_{t}]=z.
\end{split}
\end{equation*}
 This Lie superalgebra is called  a twisted Heisenberg superalgebra, denoted by $H_{n,m}^{\lambda,\kappa}$.
\end{definition}
Clearly,  Lie superalgebra $H_{n,m}^{\lambda,\kappa}$ is not nilpotent, but it is solvable, since $[H_{n,m}^{\lambda,\kappa},H_{n,m}^{\lambda,\kappa}]=H_{n,m}$.

A twisted Heisenberg superalgebra is actually  a semi-direct product of a Heisenberg  superalgebra and an even derivation.
Note that the Heisenberg  superalgebra $H_{n,m}$ has an even derivation $u$ such that
\begin{equation}\label{twisted deri}
u(z)=0, ~~u(e_{i})=\lambda_{i}\widehat{e_{i}},~~u(\widehat{e_{i}})=-\lambda_{i}e_{i}
\end{equation}
for all $i\leq k$.
Define a Lie superalgebra $H_{n,m}^{u}:= H_{n,m}\rtimes\mathbb{C}u$.
Then it is clear that the even part of $H_{n,m}^{u}$ is a twisted Heisenberg algebra (see \cite[Definition 4]{Cal2}).
\begin{proposition}
Let $H_{n,m}$ be a Heisenberg superalgebra, where $n=2k+1$ is a positive odd number and $k,m\in \mathbb{Z}_{\geq 0}$.  Let $u$ be an even derivation of $H_{n,m}$ satisfying Eq. (\ref{twisted deri}) for some $\lambda=(\lambda_{1},\ldots,\lambda_{k})\in (\mathbb{C^{\times}})^{k}$. Then $H_{n,m}^{u}$ is isomorphic to  $H_{n,m}^{\lambda,\kappa}$ for  some $\kappa=(\kappa_{1},\ldots,\kappa_{r})\in (\mathbb{C^{\times}})^{r}$, where $r\in \mathbb{Z}_{\geq0}$ and $2r\leq m.$
\end{proposition}
\begin{proof}
Recall that Eq. (\ref{heisenberg basis}) is a homogeneous basis of $H_{n,m}$. Suppose $u(w_{i})=\sum\limits_{j=1}^{m}a_{ji}w_{j},$ where all $a_{ji}\in \mathbb{C}.$   Since $u$ is a derivation, we have $a_{ij}=-a_{ji}$ for all $1\leq i,j\leq m$. Then the matrix $A=(a_{ij})$ of the restriction of $u$ to the odd part of $H_{n,m}$ with respect to the ordered basis $\{w_{1},\ldots,w_{m}\}$ is a skew symmetric matrix. By \cite[Corollary 4.4.19]{Horn}, $A$ is unitarily congruent to a matrix of the following form:
$$A'=\left(
  \begin{array}{cccccccc}
    0 & -\kappa_{1} &  &  &  &  &  &  \\
    \kappa_{1} & 0 &  &  &  &  &  &  \\
     &  & \ddots &  &  &  &  &  \\
     &  &  & 0 & -\kappa_{r} &  &  &  \\
     &  &  & \kappa_{r} & 0 &  &  &  \\
     &  &  &  &  & 0 &  &  \\
     &  &  &  &  &  & \ddots &  \\
     &  &  &  &  &  &  & 0 \\
  \end{array}
\right),$$where $\kappa_{i}\in \mathbb{C}, 1\leq i\leq r.$ By \cite[Theorem 13]{Rod}, we know that $H_{n,m}^{u}$ is isomorphic to $H_{n,m}^{\lambda,\kappa}$ for  $\kappa=(\kappa_{1},\ldots,\kappa_{r})\in (\mathbb{C^{\times}})^{r}$.
\end{proof}

\section{Fine grading on $H_{n,m}^{\lambda,\kappa}$}
\subsection{Torality and basic examples}\label{section first grading}
We are now dealing with two types of fine gradings on $H_{n,m}^{\lambda,\kappa}$, which will be relevant to our work.
For each $s$ such that $0\leq 2s\leq m-2r$, write
\begin{equation}\label{odd ker}
\begin{split}
&p_{i}:=\frac{1}{\sqrt{2}}(\eta_{2i-1}+\mathbf{i}\eta_{2i}), \\
   &q_{i}:=\frac{1}{\sqrt{2}}(\eta_{2i-1}-\mathbf{i}\eta_{2i}),\\
   &z_{j}:=\eta_{j+2s},
\end{split}
\end{equation}
where $i\leq s$, $j\leq m-2r-2s$ and $\mathbf{i}=\sqrt{-1}\in \mathbb{C}.$
Then
\begin{eqnarray*}
 [p_{i},q_{i}]=[z_{j},z_{j}]=z,
\end{eqnarray*}
for $i\leq s$ and $j\leq m-2r-2s.$
Obviously
\begin{equation}\label{second basis}
\{z,u,e_{1},\widehat{e_{1}},\ldots,e_{k},\widehat{e_{k}}\mid w_{1},\widehat{w_{1}},\ldots,w_{r},\widehat{w_{r}},p_{1},q_{1},\ldots,p_{s},q_{s},z_{1},\ldots,z_{m-2r-2s}\}
\end{equation} is a basis of $H_{n,m}^{\lambda,\kappa}$.
Then a fine grading on $H_{n,m}^{\lambda,\kappa}$ is obviously provided by our basis:
 for each $s$ such that $0\leq 2s\leq m-2r,$  we have a $G$-grading given by
\begin{eqnarray*}
  \Gamma_{1}^{s}: &&(H_{n,m}^{\lambda,\kappa})_{(2;\overline{2}; \overline{0},\ldots,\overline{0};\overline{0},\ldots,\overline{0};0,\ldots,0;\overline{0},\ldots,\overline{0})}=\langle z\rangle, \\
  &&(H_{n,m}^{\lambda,\kappa})_{(0;\overline{2};\overline{0},\ldots,\overline{0};\overline{0},\ldots,\overline{0};0,\ldots,0;\overline{0},\ldots,\overline{0})}=\langle u\rangle, \\
  &&(H_{n,m}^{\lambda,\kappa})_{(1;\overline{0};\overline{0},\ldots,\overline{1},\ldots,\overline{0};\overline{0},\ldots,\overline{0};0,\ldots,0;\overline{0},\ldots,\overline{0})}=\langle e_{i}\rangle ~(\overline{1}~~ \hbox{in the}~ i\hbox{-th slot}),\\
  &&(H_{n,m}^{\lambda,\kappa})_{(1;\overline{2};\overline{0},\ldots,\overline{1},\ldots,\overline{0};\overline{0},\ldots,\overline{0};0,\ldots,0;\overline{0},\ldots,\overline{0})}=\langle \widehat{e_{i}}\rangle, \\
  &&(H_{n,m}^{\lambda,\kappa})_{(1;\overline{1};\overline{0},\ldots,\overline{0};\overline{0},\ldots,\overline{1},\ldots,\overline{0};0,\ldots,0;\overline{0},\ldots,\overline{0})}=\langle w_{j}\rangle, \\
  &&(H_{n,m}^{\lambda,\kappa})_{(1;\overline{3};\overline{0},\ldots,\overline{0};\overline{0},\ldots,\overline{1},\ldots,\overline{0};0,\ldots,0;\overline{0},\ldots,\overline{0})}=\langle \widehat{w_{j}}\rangle, \\
  &&(H_{n,m}^{\lambda,\kappa})_{(1;\overline{1};\overline{0},\ldots,\overline{0};\overline{0},\ldots,\overline{0};0,\ldots,1,\ldots0;\overline{0},\ldots,\overline{0})}=\langle p_{t}\rangle, \\
  &&(H_{n,m}^{\lambda,\kappa})_{(1;\overline{1};\overline{0},\ldots,\overline{0};\overline{0},\ldots,\overline{0};0,\ldots,-1,\ldots,0;\overline{0},\ldots,\overline{0})}=\langle q_{t}\rangle, \\
  &&(H_{n,m}^{\lambda,\kappa})_{(1;\overline{1};\overline{0},\ldots,\overline{0};\overline{0},\ldots,\overline{0};0,\ldots,0;\overline{0},\ldots,\overline{1},\ldots,\overline{0})}=\langle z_{h}\rangle,
\end{eqnarray*}
where
$$
i\leq k,\ j\leq r,\ t\leq s,\  h\leq m-2r-2s
$$
and
$$
G=\mathbb{Z}\times \mathbb{Z}_{4}\times \mathbb{Z}_{2}^{k+r}\times \mathbb{Z}^{s}\times \mathbb{Z}_{2}^{m-2r-2s}.
$$
Clearly, the grading $\Gamma_{1}^{s}$ is not toral.  Let us
find a toral fine grading. Let
\begin{equation}\label{second grading}
\begin{split}
& u_{i}:=\frac{1}{\sqrt{2\mathbf{i}}}(e_{i}+\mathbf{i}\widehat{e_{i}}), \\
   & v_{i}:=\frac{1}{\sqrt{2\mathbf{i}}}(e_{i}-\mathbf{i}\widehat{e_{i}}), \\
   & f_{j}:=\frac{1}{\sqrt{2}}(w_{j}+\mathbf{i}\widehat{w_{j}}),\\
   & g_{j}:=\frac{1}{\sqrt{2}}(w_{j}-\mathbf{i}\widehat{w_{j}}),
\end{split}
\end{equation}
where $i\leq k$ and $j\leq r$.
Then
\begin{equation}\label{second bracket}
\begin{split}
   & [u,u_{i}]=-\mathbf{i}\lambda_{i}u_{i}, \\
   & [u,v_{i}]=\mathbf{i}\lambda_{i}v_{i},\\
   & [u_{i},v_{i}]=-\lambda_{i}z,\\
   & [u,f_{j}]=-\mathbf{i}\kappa_{i}f_{j}, \\
   & [u,g_{j}]=\mathbf{i}\kappa_{i}g_{j}, \\
   & [f_{j},g_{j}]=[p_{t},q_{t}]=[z_{h},z_{h}]=z,
 \end{split}
\end{equation}
where $ i\leq k$, $j\leq r$, $t\leq s$ and $h\leq m-2r-2s$. Obviously
\begin{equation}\label{third basis}
\{z,u,u_{1},v_{1},\ldots,u_{k},v_{k}\mid f_{1},g_{1},\ldots,f_{r},g_{r},p_{1},q_{1},\ldots,p_{s},q_{s},z_{1},\ldots,z_{m-2r-2s}\}
\end{equation}
is a basis of $H_{n,m}^{\lambda,\kappa}$.
For each $s$ such that $m-2r-2s\neq1,$  we have a  fine $\mathbb{Z}^{1+k+r+s}\times \mathbb{Z}_{2}^{m-2r-2s}$-grading
\begin{eqnarray*}
  \Gamma_{2}^{s}: &&(H_{n,m}^{\lambda,\kappa})_{(2;0,\ldots,0;0,\ldots,0;0,\ldots,0;\overline{0},\ldots,\overline{0})}=\langle z\rangle, \\
  &&(H_{n,m}^{\lambda,\kappa})_{(0;0,\ldots,0;0,\ldots,0;0,\ldots,0;\overline{0},\ldots,\overline{0})}=\langle u\rangle, \\
  &&(H_{n,m}^{\lambda,\kappa})_{(1;0,\ldots,1,\ldots,0;0,\ldots,0;0,\ldots,0;\overline{0},\ldots,\overline{0})}=\langle u_{i}\rangle\ (1~ \hbox{in the}~ i\hbox{-th slot}),\\
  &&(H_{n,m}^{\lambda,\kappa})_{(1;0,\ldots,-1,\ldots,0;0,\ldots,0;0,\ldots,0;\overline{0},\ldots,\overline{0})}=\langle v_{i}\rangle, \\
  &&(H_{n,m}^{\lambda,\kappa})_{(1;0,\ldots,0;0,\ldots,1,\ldots,0;0,\ldots,0;\overline{0},\ldots,\overline{0})}=\langle f_{j}\rangle, \\
  &&(H_{n,m}^{\lambda,\kappa})_{(1;0,\ldots,0;0,\ldots,-1,\ldots,0;0,\ldots,0;\overline{0},\ldots,\overline{0})}=\langle g_{j}\rangle, \\
 &&(H_{n,m}^{\lambda,\kappa})_{(1;0,\ldots,0;0,\ldots,0;0,\ldots,1,\ldots,0;\overline{0},\ldots,\overline{0})}=\langle p_{t}\rangle, \\
  &&(H_{n,m}^{\lambda,\kappa})_{(1;0,\ldots,0;0,\ldots,0;0,\ldots,-1,\ldots,0;\overline{0},\ldots,\overline{0})}=\langle q_{t}\rangle, \\
  &&(H_{n,m}^{\lambda,\kappa})_{(1;0,\ldots,0;0,\ldots,0;0,\ldots,0;\overline{0},\ldots,\overline{1},\ldots,\overline{0})}=\langle z_{h}\rangle,
\end{eqnarray*}
where
$$
i\leq k,\ j\leq r,\ t\leq s,\  h\leq m-2r-2s.
$$
If $m$ is even, then $\Gamma_{2}^{\frac{m}{2}-r}$ is a toral fine grading. If $m$ is odd, for $s=\frac{m-1}{2}-r$, then we have a  toral fine $\mathbb{Z}^{1+k+r+s} $-grading
\begin{eqnarray*}
  \Gamma_{2}^{\frac{m-1}{2}-r}: &&(H_{n,m}^{\lambda,\kappa})_{(2;0,\ldots,0;0,\ldots,0;0,\ldots,0)}=\langle z\rangle, \\
  &&(H_{n,m}^{\lambda,\kappa})_{(0;0,\ldots,0;0,\ldots,0;0,\ldots,0)}=\langle u\rangle, \\
  &&(H_{n,m}^{\lambda,\kappa})_{(1;0,\ldots,1,\ldots,0;0,\ldots,0;0,\ldots,0)}=\langle u_{i}\rangle (1~ \hbox{in the}~ i\hbox{-th slot}),\\
  &&(H_{n,m}^{\lambda,\kappa})_{(1;0,\ldots,-1,\ldots,0;0,\ldots,0;0,\ldots,0)}=\langle v_{i}\rangle, \\
  &&(H_{n,m}^{\lambda,\kappa})_{(1;0,\ldots,0;0,\ldots,1,\ldots,0;0,\ldots,0)}=\langle f_{j}\rangle, \\
  &&(H_{n,m}^{\lambda,\kappa})_{(1;0,\ldots,0;0,\ldots,-1,\ldots,0;0,\ldots,0)}=\langle g_{j}\rangle, \\
 &&(H_{n,m}^{\lambda,\kappa})_{(1;0,\ldots,0;0,\ldots,0;0,\ldots,1,\ldots,0)}=\langle p_{t}\rangle, \\
  &&(H_{n,m}^{\lambda,\kappa})_{(1;0,\ldots,0;0,\ldots,0;0,\ldots,-1,\ldots,0)}=\langle q_{t}\rangle, \\
  &&(H_{n,m}^{\lambda,\kappa})_{(1;0,\ldots,0;0,\ldots,0;0,\ldots,0)}=\langle z_{1}\rangle.
\end{eqnarray*}
where
$i\leq k,\ j\leq r,\ t\leq s.$
\begin{lemma}\label{homo basis} For each group grading on $H_{n,m}^{\lambda,\kappa}$, there is a basis $$\{z,u',u_{1}',v_{1}',\ldots,u_{k}',v_{k}'\mid f_{1}',g_{1}',\ldots,
f_{r}',g_{r}',p_{1}',q_{1}',\ldots,
p_{s}',q_{s}',z_{1}',\ldots,z_{m-2r-2s}'\}$$ of $H_{n,m}^{\lambda,\kappa}$ with $u'$ a  homogeneous element of the grading  such that the only nonzero brackets are given by
\begin{eqnarray*}
   && [u',u_{i}']=-\mathbf{i}\lambda_{i}u_{i}', \\
   && [u',v_{i}']=\mathbf{i}\lambda_{i}v_{i}', \\
   && [u_{i}',v_{i}']=-\lambda_{i}z,\\
   && [u',f_{j}']=-\mathbf{i}\kappa_{i}f_{j}',\\
   && [u',g_{j}']=\mathbf{i}\kappa_{i}g_{j}', \\
   && [f_{j}',g_{j}']=[p_{t}',q_{t}']=[z_{h}',z_{h}']=z,
\end{eqnarray*}
where $i\leq k$, $j\leq r$, $t\leq s$ and $h\leq m-2r-2s$.
\end{lemma}
\begin{proof}
Let $\Gamma:L=\bigoplus_{g\in G}L_{g}$ be a group grading on $L=H_{n,m}^{\lambda,\kappa}$ and  $L_{\overline{0}}$ denote the even part of $L$. Since any automorphism leaves invariant  $[L_{\overline{0}},L_{\overline{0}}]=(H_{n,m})_{\overline{0}}$ and $\langle z\rangle$,
one sees that $z$ is homogeneous and $(H_{n,m})_{\overline{0}}$ is a  $G$-graded subspace.

Since not every homogeneous element of $L_{\overline{0}}$ is contained in $(H_{n,m})_{\overline{0}}$, we can take a homogeneous element $a=\eta u+h\in L_{\overline{0}}$ with $\eta$ a nonzero scalar and $h\in (H_{n,m})_{\overline{0}}$. Then $u':=\eta^{-1}a$ is a homogeneous element and $u'-u\in (H_{n,m})_{\overline{0}}$. Recall that Eq. (\ref{third basis}) is a basis of $L$. Then $$u'=u+\alpha z+\sum\limits_{i=1}^{k}\alpha_{i}u_{i}+\sum\limits_{i=1}^{k}\beta_{i}v_{i}$$ for some  scalars $\alpha,\alpha_{i},\beta_{i}\in \mathbb{C}.$ Take $u_{i}'=u_{i}+\mathbf{i}\beta_{i}z$, $v_{i}'=v_{i}+\mathbf{i}\alpha_{i}z$, $f_{i}'=f_{i}$, $g_{i}'=g_{i}$, $p_{i}'=p_{i}$, $q_{i}'=q_{i}$, $z_{i}'=z_{i}$. Then the basis $$\{z,u',u_{1}',v_{1}',\ldots,u_{k}',v_{k}'\mid f_{1}',g_{1}',\ldots,
f_{r}',g_{r}',p_{1}',q_{1}',\ldots,p_{s}',q_{s}',z_{1}',\ldots,z_{m-2r-2s}'\}$$ satisfies the required conditions.
\end{proof}
 Lemma \ref{homo basis} tells us a useful fact: for any group grading $\Gamma:L=\bigoplus_{g\in G}L_{g}$  on $L=H_{n,m}^{\lambda,\kappa}$, the element $u$ in Eq. (\ref{third basis}) can always be assumed to be a homogeneous element of the grading.
  Let us denote by $\mathrm{deg}u=h\in G$ the degree of $u$ in $\Gamma$. Now we will show that $h$ is necessarily of finite order. Denote $$\varphi:=\mathrm{ad}u:L\longrightarrow L,~~~~~x\mapsto[u,x].$$
If $0\neq x\in [u,L]$ is a homogeneous element, then there is $g\in G$ such that $$x=\sum\limits_{i=1}^{k}(c_{i}u_{i}+d_{i}v_{i})+\sum\limits_{i=1}^{r}(c_{i}'f_{i}+d_{i}'g_{i})\in L_{g}$$ for some scalars $c_{i},d_{i},c_{i}',d_{i}'\in \mathbb{C}.$ Then $$\varphi^{t}(x)=\sum_{i=1}^{k}((-1)^{t}c_{i}u_{i}+d_{i}v_{i})(\mathbf{i}\lambda_{i})^{t}+\sum_{i=1}^{r}((-1)^{t}c_{i}'f_{i}+d_{i}'g_{i})(\mathbf{i}\kappa_{i})^{t}\in L_{g+th}$$
is not zero for all $t\in \mathbb{N}$. Since there are not more than $2k+2r$ linearly independent elements  in the set\begin{eqnarray*}
                                                                                                                           &&\left\{\sum_{i=1}^{k}((-1)^{t}c_{i}u_{i}+d_{i}v_{i})(\mathbf{i}\lambda_{i})^{t}+\sum_{i=1}^{r}((-1)^{t}c_{i}'f_{i}+d_{i}'g_{i})(\mathbf{i}\kappa_{i})^{t}\mid t\in \mathbb{Z}_{\geq0}\right\} \\
                                                                                                                           &&\subset \langle\{u_{1},v_{1},\ldots,u_{k},v_{k}\mid f_{1},g_{1},\ldots,f_{r},g_{r}\}\rangle,
                                                                                                                       \end{eqnarray*}
 there is a positive integer $p\leq 2k+2r+1$ with $\varphi^{p}(x)\in \langle\{\varphi^{t}(x)\mid 0\leq t< p\}\rangle.$
Since $\varphi^{p}(x)\in L_{g+ph}\cap(\sum_{t<p}L_{g+th})$, there exists $t<p$ such that $g+ph=g+th.$ Hence $(p-t)h=0$, as desired.

Let us denote by $l$ the order of $h$ in $G$. By Eq. (\ref{second bracket}), we know the set of eigenvalues of $\varphi$ is $$\{-\mathbf{i}\lambda_{1},\mathbf{i}\lambda_{1},\ldots,-\mathbf{i}\lambda_{k},\mathbf{i}\lambda_{k},-\mathbf{i}\kappa_{1},\mathbf{i}\kappa_{1},\ldots,
-\mathbf{i}\kappa_{r},\mathbf{i}\kappa_{r},0,\ldots,0\}$$ with respective eigenvectors$$\{u_{1},v_{1},\ldots,u_{k},v_{k}, f_{1},g_{1},\ldots,f_{r},g_{r},u,z,p_{1},q_{1},\ldots,p_{s},q_{s},z_{1},\ldots,z_{m-2r-2s}\}.$$  Then the set of eigenvalues of $\varphi|_{[u,L]}$ is $$\{-\mathbf{i}\lambda_{1},\mathbf{i}\lambda_{1},\ldots,-\mathbf{i}\lambda_{k},\mathbf{i}\lambda_{k},-\mathbf{i}\kappa_{1},\mathbf{i}\kappa_{1},\ldots,
-\mathbf{i}\kappa_{r},\mathbf{i}\kappa_{r}\}=:\mathrm{Spec}(u).$$
Set $$\mu_{i}:=\left\{
                \begin{array}{ll}
                  -\mathbf{i}\lambda_{i}, & 1\leq i\leq k \\
                  -\mathbf{i}\kappa_{i-k}, & k+1\leq i\leq k+r.
                \end{array}
              \right.
$$Then $$\mathrm{Spec}(u)=\{\mu_{1},-\mu_{1},\ldots,\mu_{k+r},-\mu_{k+r}\}.$$
Fix  $\mu_{i}\in \mathrm{Spec}(u)$, consider the eigenspace of $\varphi$ given by $V_{\mu_{i}}:=\{x\in L\mid \varphi(x)=\mu_{i} x\}$ and denote by $V_{\mu_{i}}^{l}:=\{x\in L\mid\varphi^{l}(x)=\mu_{i}^{l}x\}.$ It is obviously that $V_{\mu_{i}}\subset V_{\mu_{i}}^{l}$ is nonzero. Moreover, as $V_{\mu_{i}}^{l}$ is invariant under $\varphi$, one sees that $\varphi|_{V_{\mu_{i}}^{l}}$ is diagonalizable and
\begin{equation}\label{eigenspace}
  V_{\mu_{i}}^{l}=\bigoplus_{j=0}^{l-1}V_{\xi^{j}\mu_{i}}
\end{equation}
for $\xi$ a fixed primitive $l$th root of  unit. Note that if $x\in V_{\mu_{i}}^{l}$, then $$\sum\limits_{q=0}^{l-1}(\xi^{l-j}\mu_{i}^{-1})^{q}\varphi^{q}(x)\in V_{\xi^{j}\mu_{i}}$$ for any $j=0,\ldots,l-1$.

Recall that if $f\in \mathrm{End}(L)$ satisfies $f(L_{g})\subset L_{g}$ for all $g\in G$, then for each $\alpha\in \mathbb{C}$, the eigenspace $V_{\alpha}=\{x\in L\mid f(x)=\alpha x\}$ is $G$-graded. Since $\varphi^{l}(L_{g})\subset L_{g+lh}=L_{g}$,  the eigenspace  $V_{\mu_{i}}^{l}$ of $\varphi^{l}$ is a $G$-graded subspace of $L$. Thus we can take $0\neq x\in V_{\mu_{i}}^{l}\cap L_{g}$ for some $g\in G.$ For each $j=0,\ldots,l-1$, the element
$$\sum\limits_{q=0}^{l-1}(\xi^{l-j}\mu_{i}^{-1})^{q}\varphi^{q}(x)\in \sum_{q=0}^{l-1}L_{g+qh}$$ must be nonzero. Consequently, $V_{\xi^{j}\mu_{i}}\neq0$ for all $j$. Then $$\{\xi^{j}\mu_{i}\mid j=0,\ldots,l-1\}\subset \mathrm{Spec}(u)$$ for any $\mu_{i}\in \mathrm{Spec}(u)$. Hence
\begin{equation}\label{spec}
  \mathrm{Spec}(u)=\{\pm\xi^{j}\mu_{i}\mid 0\leq j<l,1\leq i\leq k+r\}.
\end{equation}

As mentioned above, our aim is to classify all the fine gradings up to equivalence and characterize their Weyl groups for twisted Heisenberg superalgebras.  Now we are in position to do that in a special case, while the general case is left to  Sections \ref{general case} and \ref{sec.4.3}.

\begin{theorem}\label{classification1}
Suppose that $\frac{\lambda_{i}}{\lambda_{j}}$ and $\frac{\kappa_{i}}{\kappa_{j}}$ are not  roots of unit for all $i\neq j$.

$(\mathrm{1})$ Up to equivalence, there are $m-2r+2$ fine gradings on $H_{n,m}^{\lambda,\kappa}$ if $m$ is even and $m-2r+1$ in case $m$ is odd, namely, $$\{\Gamma_{1}^{s},\Gamma_{2}^{s}:0\leq2s\leq m-2r\},$$ and only $\Gamma_{2}^{\frac{m}{2}-r}$($m$ is even) and $\Gamma_{2}^{\frac{m-1}{2}-r}$($m$ is odd) are toral.

$(\mathrm{2})$ The Weyl groups of these fine gradings are $$\mathcal{W}(\Gamma_{1}^{s})\cong\mathbb{Z}_{2}^{k+r+s}\rtimes(S_{s}\times S_{m-2r-2s})$$ and $$\mathcal{W}(\Gamma_{2}^{s})\cong \mathbb{Z}_{2}^{1+s}\rtimes(S_{s}\times S_{m-2r-2s}).$$
\end{theorem}

\begin{proof}
(1) Let $\Gamma:L=\bigoplus_{g\in G}L_{g}$ be a group grading on $L=H_{n,m}^{\lambda,\kappa}$. By Lemma \ref{homo basis}, we can assume that $u$ is homogeneous with  $\mathrm{deg}u=h\in G$ of finite order $l$. Now we will show that $l=1$ or $l=2$. Otherwise, take $\xi$ a primitive $l$th root of unit. As $\xi\mu_{1}\neq\pm\mu_{1}$, by Eq. (\ref{spec}), there is  $1\neq i\leq k$ such that $\xi\mu_{1}\in \pm\mu_{i}$. Then either $(\frac{\mu_{i}}{\mu_{1}})^{l}=1$ or $(\frac{\mu_{i}}{\mu_{1}})^{2l}=1$, this is a contradiction. Hence, we distinguish two cases.

(i) First consider $l=1$, that is $h=0\in G.$ Then $\varphi(L_{g})\subset L_{g}$ for all $g\in G.$ So we can take a basis of homogeneous elements which are eigenvectors for $\varphi$.   Recall that the eigenvalues of $\varphi|_{[u,L]}$ consist of $\{\pm\mu_{1},\ldots,\pm\mu_{k+r}\}.$
Let $L_{\overline{0}}$ and $L_{\overline{1}}$ denote the even part and the odd part of $L$ respectively. Take a homogeneous element $x_{1}\neq0$ in $V_{\mu_{1}}\cap L_{\overline{0}}.$ As $[x_{1},V_{-\mu_{1}}\cap L_{\overline{0}}]\neq0$, there is some element $y_{1}\in V_{-\mu_{1}}\cap L_{\overline{0}}$  such that $[x_{1},y_{1}]=-\lambda_{1}z.$ Now $$[u,L]=W\oplus \mathrm{Z}_{[u,L]}(W)$$ for $W:=\langle x_{1},y_{1}\rangle$,
where $W$ and its centralizer $\mathrm{Z}_{[u,L]}(W)$ are $G$-graded and $\varphi$-invariant. We continue by induction until finding a homogeneous basis $$\{x_{1},y_{1},\ldots,x_{k},y_{k}\mid\omega_{1},\nu_{1},\ldots,\omega_{r},\nu_{r}\}$$ of $[u,L]$ such that $[x_{i},y_{i}]=-\lambda_{i}z$, $[u,x_{i}]=\mu_{i}x_{i}=-\mathbf{i}\lambda_{i}x_{i}$, $[u,y_{i}]=-\mu_{i}y_{i}=\mathbf{i}\lambda_{i}y_{i}$ for $ i\leq k$ and $[\omega_{i},\nu_{i}]=z$, $[u,\omega_{i}]=-\mathbf{i}\kappa_{i}\omega_{i}$, $[u,\nu_{i}]=\mathbf{i}\kappa_{i}y_{i}$ for $i\leq r$. Note that $$L=\langle z\rangle\oplus\langle u\rangle\oplus[u,L]\oplus V,$$ where $V:=\ker \varphi\cap L_{\overline{1}}.$ Then $V$ is a $m-2r$-dimensional linear space with a symmetric nondegenerate bilinear form $\langle\cdot,\cdot\rangle: V\times V\rightarrow \mathbb{C}$ such that $[x,y]=\langle x,y\rangle z$ for all $x,y\in V.$ Since $V$ is a $G$-graded subspace, we can assume that $V=\bigoplus_{g\in G} V_{g}$, where $V_{g}:=V\cap L_{g}$.
By \cite[Lemma 2]{Cal2}, there is a  basis
$B=\{\gamma_{1},\rho_{1},\ldots,\gamma_{s},\rho_{s},\delta_{1},\ldots,\delta_{m-2r-2s}\}$ of $V$ such that

 \begin{itemize}
   \item $B\subset\bigcup_{g}V_{g};$~~~~~~~~~~~~~~~~~~~~~~~~~~~~~~~~~~~~~~~~~~~~~~~~~~~~~~~~~~~~~~~~~~~~~~~~~~~~~~~~~~~~~~~~~~~~~~~~~~$(*)$
   \item $\langle\gamma_{i},\rho_{i}\rangle=1,~~\langle\delta_{i},\delta_{i}\rangle=1;$
   \item $\hbox{any other inner product of elements in}$ $B$ $\hbox{is zero.}$
 \end{itemize}

Then   the map $u\mapsto u$, $z\mapsto z$, $x_{i}\mapsto u_{i}$, $y_{i}\mapsto v_{i}$, $\omega_{i}\mapsto f_{i}$, $\nu_{i}\mapsto g_{i}$, $\gamma_{i}\mapsto p_{i}$, $\rho_{i}\mapsto q_{i}$ and $\delta_{i}\mapsto z_{i}$ extends to a Lie superalgebra isomorphism which applies $\Gamma$ into a coarsening of
$\Gamma_{2}^{s}$.

(ii) Second consider the case $l=2$, that is, $2h=0$ but $h\neq0$. Then $\varphi^{2}$ preserves the grading $\Gamma$ and it is diagonalizable with eigenvalues $$\{\mu_{1}^{2},\ldots,\mu_{k}^{2},\mu_{k+1}^{2},\ldots,\mu_{k+r}^{2},0,\ldots,0\}.$$ It is clear that $\varphi$ applies $$V_{\mu_{i}}^{2}=\{x\in L\mid \varphi^{2}(x)=\mu_{i}^{2}x\}$$ into itself. Since $\varphi^{2}$ preserves the grading, we have $V_{\mu_{i}}^{2}$ is $G$-graded for each $i$. For any $0\neq x_{1}\in V_{\mu_{1}}^{2}\cap L_{g}\cap L_{\overline{0}}$ a homogeneous element of the grading, $\varphi(x_{1})$ is independent with $x_{1}.$ Take $y_{1}=\frac{1}{\lambda_{1}}\varphi(x_{1}).$ Then $\varphi(y_{1})=-\lambda_{1}x_{1}$.  Since $\lambda_{i}^{2}\neq \lambda_{j}^{2}$ for all  $i\neq j$, we have $\dim V_{\mu_{i}}^{2}\cap L_{\overline{0}}=2$ for all $ i\leq k$. So we have $[x_{1},y_{1}]\neq0$ and $V_{\mu_{1}}^{2}\cap L_{\overline{0}}=\langle x_{1},y_{1}\rangle$. Since $\mathbb{C}$ is algebraically closed, we can scale to get $[x_{1},y_{1}]=\lambda_{1}z$.  Then $$[u,L_{\overline{0}}]=W\oplus \mathrm{Z}_{[u,L_{\overline{0}}]}(W)$$ for $W:=\langle x_{1},y_{1}\rangle$, where  $W$ and $\mathrm{Z}_{[u,L_{\overline{0}}]}(W)$ are $G$-graded and $\varphi$-invariant. We continue  this process on $\mathrm{Z}_{[u,L_{\overline{0}}]}(W)$ until finding a homogeneous basis $$\{x_{1},y_{1},\ldots,x_{k},y_{k}\}$$ of $[u,L_{\overline{0}}]$ such that $[x_{i},y_{i}]=\lambda_{i}z$, $[u,x_{i}]=\lambda_{i}y_{i}$, $[u,y_{i}]=-\lambda_{i}x_{i}$.

For any $0\neq \omega_{1}\in V_{\mu_{k+1}}^{2} \cap L_{g} \cap L_{\overline{1}}$, $\varphi(\omega_{1})$ is independent with $\omega_{1}.$ Take $\nu_{1}=\frac{1}{\kappa_{1}}\varphi(\omega_{1}).$ Then $\varphi(\nu_{1})=-\kappa_{1}\omega_{1}$. Since $\frac{\kappa_{i}}{\kappa_{j}}$ is not a root of unit for all $i\neq j$, we have $\dim V_{\mu_{k+1}}^{2}\cap L_{\overline{1}}=2$  and $$V_{\mu_{k+1}}^{2}\cap L_{\overline{1}}=\langle \omega_{1},\nu_{1}\rangle.$$ As before, we have $[\omega_{1},\omega_{1}]\neq 0$ and $[\nu_{1},\nu_{1}]\neq0$. Hence we can scale to get $[\omega_{1},\omega_{1}]=z$ and $[\nu_{1},\nu_{1}]=z$. Then  $$[u,L_{\overline{1}}]=W'\oplus \mathrm{Z}_{[u,L_{\overline{1}}]}(W')$$ for $W':=\langle \omega_{1},\nu_{1}\rangle$, where both $W'$ and $\mathrm{Z}_{[u,L_{\overline{1}}]}(W')$ are $G$-graded and $\varphi$-invariant. We continue  this process on $\mathrm{Z}_{[u,L_{\overline{1}}]}(W')$ until finding a basis of homogeneous elements $$\{\omega_{1},\nu_{1},\ldots,\omega_{k},\nu_{k}\}$$ of $[u,L_{\overline{1}}]$ such that
 $[u,\omega_{i}]=\kappa_{i}\nu_{i}$, $[u,\nu_{i}]=-\kappa_{i}\omega_{i}$ and $[\omega_{i},\omega_{i}]=[\nu_{i},\nu_{i}]=z$. Similar to the first case, there is a homogeneous basis $$B=\{\gamma_{1},\rho_{1},\ldots,\gamma_{s},\rho_{s},\delta_{1},\ldots,\delta_{m-2r-2s}\}$$ of $V:=\mathrm{ker}\varphi\cap L_{\overline{1}}$ satisfying  $(*)$. Then  the map $u\mapsto u$, $z\mapsto z$, $x_{i}\mapsto e_{i}$, $y_{i}\mapsto \widehat{e_{i}}$, $\omega_{i}\mapsto w_{i}$, $\nu_{i}\mapsto \widehat{w_{i}}$, $\gamma_{i}\mapsto p_{i}$, $\rho_{i}\mapsto q_{i}$ and $\delta_{i}\mapsto z_{i}$ extends to a Lie superalgebra isomorphism which applies $\Gamma$ into a coarsening of
$\Gamma_{1}^{s}$.

(2) Now we compute the Weyl groups of these fine gradings. Note that for any $f\in \mathrm{Aut}(L)$, we have $0\neq f(z)\in \langle z\rangle$.
Denote by $[f]$ the class of an automorphism $f \in \mathrm{Aut}(\Gamma)$ in the quotient $\mathcal{W}(\Gamma)$.

(i) Suppose $f\in \mathrm{Aut}(\Gamma_{1}^{s})$. Then $f(u)\in \langle u\rangle$. Otherwise, there  exists some $i\leq k$ such that either $f(e_{i})\in \langle u\rangle$ or $f(\widehat{e_{i}})\in \langle u\rangle$, hence $$0\neq f(\lambda_{i}z)=[f(e_{i}),f(\widehat{e_{i}})]\in [u,L]\cap \langle z\rangle=0.$$
Then  $$f(e_{i}),f(\widehat{e_{i}})\in \langle e_{1}\rangle\cup\cdots\cup\langle e_{k}\rangle\cup\langle \widehat{e_{1}}\rangle\cup\cdots\cup\langle \widehat{e_{k}}\rangle$$ for all $i\leq k$.
 Since $$0\neq f([u,w_{i}])=[f(u),f(w_{i})]$$ and $$0\neq f([u,\widehat{w_{i}}])=[f(u),f(\widehat{w_{i}})],$$ we know $f(w_{i}),~f(\widehat{w_{i}})\not\in \ker (\mathrm{ad}u)$ for any $i=1,\ldots,r$, so we have $$f(w_{i}),f(\widehat{w_{i}})\in \langle w_{1}\rangle\cup\cdots\cup\langle w_{r}\rangle\cup\langle \widehat{w_{1}}\rangle\cup\cdots\cup\langle \widehat{w_{r}}\rangle.$$
Similarly we have $$f(p_{i}),f(q_{i})\in \langle p_{1}\rangle\cup\cdots\cup\langle p_{s}\rangle\cup\langle q_{1}\rangle\cup\cdots\cup\langle q_{s}\rangle$$ and $$f(z_{i})\in \langle z_{1}\rangle\cup\cdots\cup\langle z_{m-2r-2s}\rangle.$$
For each index $i\leq k$, there exists a unique $\theta_{i}\in\mathrm{Aut}(\Gamma_{1}^{s})$ such that $\theta_{i}|_{L_{\overline{1}}}=\mathbf{id},$
$$\theta_{i}(e_{j})=e_{j},~~~ \theta_{i}(\widehat{e_{j}})=\widehat{e_{j}}$$ for each $j\neq i$ and
$$\theta_{i}(z)=z,~~~\theta_{i}(u)=u,~~~\theta_{i}(e_{i})=\widehat{e_{i}},~~~ \theta_{i}(\widehat{e_{i}})=-e_{i}.$$
  Hence we can compose $f$ with some $\theta_{i}$'s if necessary to obtain that $$f':= \theta_{i_{1}}\cdots \theta_{i_{t}}f\in \mathrm{Aut}(\Gamma_{1}^{s})$$ satisfies
$$f'(e_{i})\in \langle e_{1}\rangle\cup\langle e_{2}\rangle\cup\cdots\cup\langle e_{k}\rangle$$ and
 $$f'(\widehat{e_{i}})\in \langle \widehat{e_{1}}\rangle\cup\langle \widehat{e_{2}}\rangle\cup\cdots\cup\langle \widehat{e_{k}}\rangle$$for each $i\leq k.$ Thus, there is $\sigma\in S_{k}$ (the permutation group of $k$ elements) such that $$f'(z)=\alpha z,~~~ f'(u)=\beta u,~~~ f'(e_{i})=\alpha_{i}e_{\sigma(i)},~~~ f'(\widehat{e_{i}})=\beta_{i}\widehat{e_{\sigma(i)}}$$ for any $i\leq k,$ with $\alpha,\beta,\alpha_{i},\beta_{i}\in \mathbb{C}^{\times}$.
From here, the equality $$f'([u,e_{i}])=[f'(u),f'(e_{i})]$$ implies that $\beta_{i}\lambda_{i}=\beta\alpha_{i}\lambda_{\sigma(i)}.$
Similarly, we have $\alpha_{i}\lambda_{i}=\beta\beta_{i}\lambda_{\sigma(i)}.$ Thus we get $\lambda_{\sigma(i)}\in \pm\beta^{-1}\lambda_{i}$ for any $i\leq k.$ By multiplying, we have $$\prod_{i=1}^{k}\lambda_{\sigma(i)}\in \pm\beta^{-k}\prod_{i=1}^{k}\lambda_{i}.$$ Hence $\beta^{2k}=1.$ As $\frac{\lambda_{\sigma(i)}}{\lambda_{i}}$ is not a root of unit if $\sigma(i)\neq i,$ we conclude that $\sigma=\mathbf{id}$.

For each index $i\leq r$, there exists a unique $\vartheta_{i}\in\mathrm{Aut}(\Gamma_{1}^{s})$ such that $\vartheta_{i}|_{L_{\overline{0}}}=\mathbf{id}$, $$\vartheta_{i}(w_{i})=\mathbf{i}\widehat{w_{i}},~~~ \vartheta_{i}(\widehat{w_{i}})=-\mathbf{i}w_{i},~~~\vartheta_{i}(p_{j})=p_{j},~~~\vartheta_{i}(q_{j})=q_{j},~~~\vartheta_{i}(z_{h})=z_{h}$$ for all $j\leq s$,
$h\leq m-2r-2s$ and $$\vartheta_{i}(w_{j})=w_{j},~~~ \vartheta_{i}(\widehat{w_{j}})=\widehat{w_{j}}$$ for each $j\neq i$.
       Hence we can compose $f'$ with some $\vartheta_{i}$'s if necessary to obtain that $f'':= \vartheta_{j_{1}}\cdots \vartheta_{j_{h}}f'\in \mathrm{Aut}(\Gamma_{1}^{s})$ satisfies
$$f''(w_{i})\in \langle w_{1}\rangle\cup\langle w_{2}\rangle\cup\cdots\cup\langle w_{r}\rangle$$ and
 $$f''(\widehat{w_{i}})\in \langle \widehat{w_{1}}\rangle\cup\langle \widehat{w_{2}}\rangle\cup\cdots\cup\langle \widehat{w_{r}}\rangle$$
for each $i\leq r.$ Thus, there is $\sigma'\in S_{r}$ such that $$f''(z)=\alpha z, ~~~f''(u)=\beta u,~~~ f''(w_{i})=\alpha_{i}'w_{\sigma'(i)},~~~ f''(\widehat{w_{i}})=\beta_{i}'\widehat{w_{\sigma'(i)}}$$ for any $i\leq r,$ with $\alpha_{i}',\beta_{i}'\in \mathbb{C}^{\times}$. By a similar argument as the above, we have $\sigma'=\mathbf{id}$.

For each index $i\leq s$, there exists a unique $\varrho_{i}\in\mathrm{Aut}(\Gamma_{1}^{s})$ such that $\varrho_{i}|_{L_{\overline{0}}}=\mathbf{id},$ $$\varrho_{i}(w_{j})=w_{j},~~~\varrho_{i}(\widehat{w_{j}})=\widehat{w_{j}},~~~\varrho_{i}(p_{i})=q_{i}, ~~~\varrho_{i}(q_{i})=p_{i},~~~\varrho_{i}(z_{t})=z_{t}$$ for all $ j\leq r,$ $t\leq m-2r-2s$ and $$\varrho_{i}(p_{j})=p_{j}, ~~~\varrho_{i}(q_{j})=q_{j}$$ for each $j\neq i$.
 Hence we can compose $f''$ with some $\varrho_{i}$'s if necessary to obtain that $f''':= \varrho_{k_{1}}\cdots \varrho_{k_{l}}f''$ satisfies
$$f'''(p_{i})\in \langle p_{1}\rangle\cup\langle p_{2}\rangle\cup\cdots\cup\langle p_{s}\rangle$$ and
$$f'''(q_{i})\in \langle q_{1}\rangle\cup\langle q_{2}\rangle\cup\cdots\cup\langle q_{s}\rangle$$
for each $i\leq s.$ Thus, there exist $\widetilde{\tau},\widetilde{\varsigma}\in \mathrm{Aut}(\Gamma_{1}^{s})$ such that $$\overline{f}:=\widetilde{\tau}\widetilde{\varsigma}f'''\in \mathrm{Stab}(\Gamma_{1}^{s}),$$
where $\tau\in S_{s}$, $\varsigma\in S_{m-2r-2s}$, $\widetilde{\tau}|_{L_{\overline{0}}}=\mathbf{id}$,
$$\widetilde{\tau}(p_{i})=p_{\tau(i)},~~~ \widetilde{\tau}(q_{i})=q_{\tau(i)},~~~  \widetilde{\tau}(w_{j})=w_{j},~~~
 \widetilde{\tau}(\widehat{w_{j}})=\widehat{w_{j}},~~~\widetilde{\tau}(z_{t})=z_{t}$$ for all $i\leq s$, $j\leq r,$ $t\leq m-2r-2s$ and
 $\widetilde{\varsigma}|_{L_{\overline{0}}}=\mathbf{id}$,
 $$\widetilde{\varsigma}(z_{t})=z_{\varsigma(t)},~~~\widetilde{\varsigma}(p_{i})=p_{i},~~~\widetilde{\varsigma}(q_{i})=q_{i},~~~
 \widetilde{\varsigma}(w_{j})=w_{j},~~~\widetilde{\varsigma}(\widehat{w_{j}})=\widehat{w_{j}}$$ for all $ t\leq m-2r-2s$, $i\leq s$, $j\leq r.$

 Since $\theta_{i}\theta_{j}=\theta_{j}\theta_{i}$, $\vartheta_{i}\vartheta_{j}=\vartheta_{j}\vartheta_{i}$, $\varrho_{i}\varrho_{j}=\varrho_{j}\varrho_{i}$ and $\theta_{i}^{2}$, $\vartheta_{i}^{2}$, $\varrho_{i}^{2}\in \mathrm{Stab}(\Gamma_{1}^{s})$, we have $$\mathcal{W}(\Gamma_{1}^{s})=\{[\theta_{i_{1}}\cdots\theta_{i_{t}}\vartheta_{j_{1}}\cdots\vartheta_{j_{h}}\varrho_{k_{1}}
 \cdots\varrho_{k_{l}}\widetilde{\tau}\widetilde{\varsigma}]\mid \tau\in S_{s}, ~\varsigma\in S_{m-2r-2s}\},$$ where $ i_{1}<\cdots< i_{t}\leq k,$ $ j_{1}<\cdots< j_{h}\leq r$ and $k_{1}<\cdots< k_{l}\leq s$.
 It is clear that we can identify $\mathcal{W}(\Gamma_{1}^{s})$ with the group
 $$\mathbb{Z}_{2}^{k+r+s}\rtimes(S_{s}\times S_{m-2r-2s}).$$

 (ii) For the other case, there exists a unique $\mu\in \mathrm{Aut}(\Gamma_{2}^{s})$ such that  $$\mu(z)=z,~~~ \mu(u)=-u,~~~ \mu(u_{i})=\mathbf{i}v_{i},~~~ \mu(v_{i})=\mathbf{i}u_{i}$$ for all $i\leq k$, $$\mu(f_{i})=g_{i},~~~ \mu(g_{i})=f_{i},~~~\mu(p_{j})=p_{j},~~~\mu(q_{j})=q_{j},~~~\mu(z_{t})=z_{t}$$ for all $i\leq r$, $j\leq s$, $t\leq m-2r-2s$. Suppose $f\in \mathrm{Aut}(\Gamma_{2}^{s})$. Then there are $\alpha, \beta\in \mathbb{C}^{\times}$ such that $f(z)=\alpha z$ and $f(u)=\beta u$. If $f(u_{i})$ is a multiple of either $u_{j}$ or $v_{j}$ for some $j$, this clearly implies that $f(v_{i})$ also is, hence there is $\sigma\in S_{k}$ such that $$f(u_{i})\in \langle u_{\sigma(i)}\rangle\cup\langle v_{\sigma(i)}\rangle$$ for all $i\leq k$. Similarly, there is $\sigma'\in S_{r}$ such that $f(f_{i})\in \langle f_{\sigma'(i)}\rangle\cup\langle g_{\sigma'(i)}\rangle$ for all $i\leq k$. As the above, we can conclude that $\sigma=\mathbf{id}$ and $\sigma'=\mathbf{id}$. By composing with $\mu$ if necessary, we can assume that $f(u_{1})\in \langle u_{1}\rangle,$ which implies $\beta=1.$ If $f(u_{i})\in \langle v_{i}\rangle$ for some $i$ or $f(f_{j})\in \langle g_{j}\rangle$ for some $j$, then $\beta=-1,$ which is a contradiction, so we have $f(u_{i})\in \langle u_{i}\rangle$ for all $i$ and $f(f_{j})\in \langle f_{j}\rangle$ for all $j$.

 Hence we can compose $f$ with some $\varrho_{i}$'s if necessary to obtain that $$f':= \varrho_{k_{1}}\cdots \varrho_{k_{t}}f\in \mathrm{Aut}(\Gamma_{2}^{s})$$ satisfies
$$f'(p_{i})\in \langle p_{1}\rangle\cup\langle p_{2}\rangle\cup\cdots\cup\langle p_{s}\rangle$$ and
$$f'(q_{i})\in \langle q_{1}\rangle\cup\langle q_{2}\rangle\cup\cdots\cup\langle q_{s}\rangle$$
for each $i=1,\ldots,s.$ Thus, there are $\tau\in S_{s}$, $\varsigma\in S_{m-2r-2s}$ such that $$\overline{f}:=\widetilde{\tau}\widetilde{\varsigma}f'\in \mathrm{Stab}(\Gamma_{2}^{s}),$$
So we have $$\mathcal{W}(\Gamma_{2}^{s})=\{[\mu\varrho_{k_{1}}
 \cdots\varrho_{k_{t}}\widetilde{\tau}\widetilde{\varsigma}]\mid\tau\in S_{s}, ~\varsigma\in S_{m-2r-2s}\},$$ where  $k_{1}<\cdots< k_{t}\leq s$.
 It is clear that we can identify $\mathcal{W}(\Gamma_{2}^{s})$ with the group
 $$\mathbb{Z}_{2}^{1+s}\rtimes(S_{s}\times S_{m-2r-2s}).$$
\end{proof}

\subsection{Fine gradings in the general case}\label{general case}
In the general case, the situation is much more complicated. In order to describe all the fine gradings, we previously need to give some key examples. Similar to the situation of the twisted  Heisenberg Lie algebras (see \cite[Section 6.4]{Cal2}), these examples motivate the definition of the block of type I, II and III.

For $\xi$ a primitive $l$th root of  unit, $\alpha,\beta$  nonzero scalars, we consider the twisted Heisenberg superalgebra $H_{2l+2,2l}^{\lambda,\kappa}$ corresponding to $\lambda=(\lambda_{1},\ldots,\lambda_{l})=(\xi\alpha,\xi^{2}\alpha,\ldots,\xi^{l-1}\alpha,\alpha)$ and $\kappa=(\kappa_{1},\ldots,\kappa_{l})=(\xi\beta,\xi^{2}\beta,\ldots,\xi^{l-1}\beta,\beta).$ Thus
\begin{eqnarray*}
   && [u,u_{i}]=-\mathbf{i}\xi^{i}\alpha u_{i},~~~
    [u,v_{i}]=\mathbf{i}\xi^{i}\alpha v_{i}, ~~~
    [u_{i},v_{i}]=-\xi^{i}\alpha z, \\
   && [u,f_{i}]=-\mathbf{i}\xi^{i}\beta f_{i},~~~~
    [u,g_{i}]=\mathbf{i}\xi^{i}\beta g_{i},~~~
    [f_{i},g_{i}]=z,
\end{eqnarray*}
with the definition of $u_{i}$'s, $v_{i}$'s, $f_{i}$'s and $g_{i}$'s as in Eq. (\ref{second grading}) and $1\leq i\leq l$.
Take now \begin{eqnarray*}
 && x_{j}=\sum\limits_{i=1}^{l}\xi^{ji}u_{i},~~~~~~
  y_{j}=-\frac{1}{l}\sum\limits_{i=1}^{l}(-1)^{j}\xi^{(j-1)i}v_{i}, \\
 &&  \omega_{j}=\sum\limits_{i=1}^{l}\xi^{ji}f_{i},~~~~~~
\nu_{j}=\frac{1}{l}\sum\limits_{i=1}^{l}(-1)^{j}\xi^{ji}g_{i},
         \end{eqnarray*}
 where $1\leq j\leq l$.
So we have \begin{eqnarray*}
              && [u,x_{h}]=-\mathbf{i}\alpha x_{h+1},~~~~~~~~
               [u,x_{l}]=-\mathbf{i}\alpha x_{1},~~~~~~~~
               [u,y_{h}]=-\mathbf{i}\alpha y_{h+1},\\
             &&  [u,y_{l}]=-(-1)^{l}\mathbf{i}\alpha y_{1},~~~~~~
               [x_{i},y_{j}]=(-1)^{j}\frac{\alpha}{l}(\sum\limits_{t=1}^{l}\xi^{t(i+j)})z,\\
              && [u,\omega_{h}]=-\mathbf{i}\beta \omega_{h+1},~~~~~~~~
              [u,\omega_{l}]=-\mathbf{i}\beta \omega_{1},~~~~~~~~
               [u,\nu_{h}]=-\mathbf{i}\beta \nu_{h+1},\\
             &&  [u,\nu_{l}]=-(-1)^{l}\mathbf{i}\beta \nu_{1},~~~~~~
             [\omega_{i},\nu_{j}]=(-1)^{j}\frac{1}{l}(\sum\limits_{t=1}^{l}\xi^{t(i+j)})z,
           \end{eqnarray*}
           where $1\leq h\leq l-1$ and $1\leq i,j\leq l$.
Note that $[x_{i},y_{j}]$ and  $[\omega_{i},\nu_{j}]$ are not zero if and only if $i+j=l,2l$. In such a case
\begin{equation}\label{block1}
[x_{l},y_{l}]=(-1)^{l}\alpha z,~~~ [x_{i},y_{l-i}]=(-1)^{l-i}\alpha z,~~~ [\omega_{l},\nu_{l}]=(-1)^{l}z,~~~[\omega_{i},\nu_{l-i}]=(-1)^{l-i}z
\end{equation}
 for all $i=1,\ldots,l-1$. Obviously $$\{x_{1},y_{1},\ldots,x_{l},y_{l}\mid\omega_{1},\nu_{1},\ldots,\omega_{l},\nu_{l}\}$$ is a family of independent vectors such that the only nonzero brackets are given by Eq. (\ref{block1}).

Therefore we have a fine grading on $L=H_{2l+2,2l}^{\lambda,\kappa}$ over the group $G=\mathbb{Z}^{2}\times \mathbb{Z}_{l}$ given by
\begin{equation}\label{grading example1}
\begin{split}
   & L_{(0;0;\overline{1})}=\langle u\rangle, \\
   & L_{(3;3;\overline{0})}=\langle z\rangle, \\
   & L_{(1;2;\overline{i})}=\langle x_{i}\rangle, \\
   & L_{(2;1;\overline{i})}=\langle y_{i}\rangle, \\
   & L_{(0;3;\overline{i})}=\langle \omega_{i}\rangle, \\
   & L_{(3;0;\overline{i})}=\langle \nu_{i}\rangle,
\end{split}
\end{equation}
where $1\leq i\leq l.$

If $l$ is even, there exist $\theta,\overline{\theta},\vartheta,\overline{\vartheta}\in \mathrm{Aut}(L)$ such that
\begin{eqnarray*}
   &&\theta(x_{i})=\mathbf{i}x_{i+1},~~\theta(y_{i})=\mathbf{i}y_{i-1},~~\theta(\omega_{i})=\omega_{i},~~\theta(\nu_{i})=\nu_{i},~~\theta(z)=z,~~\theta(u)=u; \\
   && \overline{\theta}(\omega_{i})=\mathbf{i}\omega_{i+1},~~\overline{\theta}(\nu_{i})=\mathbf{i}\nu_{i-1},~~\overline{\theta}(x_{i})=x_{i},~~\overline{\theta}(y_{i})=y_{i},~~\overline{\theta}(z)=z,~~\overline{\theta}(u)=u; \\
   && \vartheta(x_{i})=y_{i},~~~~~~\vartheta(y_{i})=-x_{i},~~~\vartheta(\omega_{i})=\omega_{i},~~~\vartheta(\nu_{i})=\nu_{i},~~~\vartheta(z)=z,~~~\vartheta(u)=u; \\
 && \overline{\vartheta}(x_{i})=x_{i},~~~~~~\overline{\vartheta}(y_{i})=y_{i},~~~~~\overline{\vartheta}(\omega_{i})=\nu_{i},~~~~~\overline{\vartheta}(\nu_{i})=\omega_{i},~~~~~\overline{\vartheta}(z)=z,~~~~~\overline{\vartheta}(u)=u, \end{eqnarray*}
where the indices are taken modulo $l.$ It is not difficult to check the Weyl group of the grading described in Eq. (\ref{grading example1}) is generated by the classes $[\theta],[\overline{\theta}]$ and $[\vartheta],[\overline{\vartheta}]$. Note that $\langle[\theta],[\vartheta]\rangle\cong D_{l}$, $\langle[\overline{\theta}],[\overline{\vartheta}]\rangle\cong D_{l}$,  where $D_{l}$ is the
Dihedral group. Then the Weyl group is isomorphic to  $D_{l}\times D_{l}$.

If $l$ is odd, the maps $\theta,\overline{\theta}$ are still automorphisms, but $\vartheta$ and $\overline{\vartheta}$ are not longer automorphisms. Now there exists a unique $\vartheta'\in \mathrm{Aut}(\Gamma)$ such that
$\vartheta'(u)=-u,~~\vartheta'(z)=z$ and
$$\vartheta'(x_{i})=(-1)^{i}y_{i},~~\vartheta'(y_{i})=(-1)^{i+1}x_{i},~~\vartheta'(\omega_{i})=(-1)^{i}
\nu_{i},~\vartheta'(\nu_{i})=(-1)^{i}\omega_{i}$$ for all $i\leq l$. It is not difficult to check the Weyl group of the grading described in Eq. (\ref{grading example1}) is generated by the classes $[\theta],[\overline{\theta}]$ and $[\vartheta']$, which is isomorphic to $(\mathbb{Z}_{l}\times \mathbb{Z}_{l})\rtimes \mathbb{Z}_{2}$.

This example motivates the following definition.
\begin{definition}
Let $L$ be a Lie  superalgebra, $z\in L$ a fixed even element, $u$ an arbitrary even element and $\alpha,\beta\in \mathbb{C}^{\times}$. A set $\{x_{1},y_{1},\ldots,x_{l},y_{l}\}$,  which will be called a block of type $B_{l}^{\mathrm{I}}(u,\alpha)$, or simply be called an even block of type I, is given by a family of $2l$ independent
even elements in $L$,
satisfying that the only nonzero products among them are the following:
\begin{eqnarray*}
              && [u,x_{i}]=-\mathbf{i}\alpha x_{i+1},\\
              && [u,x_{l}]=-\mathbf{i}\alpha x_{1},\\
              && [u,y_{i}]=-\mathbf{i}\alpha y_{i+1},\\
              && [u,y_{l}]=-(-1)^{l}\mathbf{i}\alpha y_{1},\\
              && [x_{i},y_{l-i}]=(-1)^{l-i}\alpha z,\\
              && [x_{l},y_{l}]=(-1)^{l}\alpha z,
           \end{eqnarray*}
           where $1\leq i\leq l-1$.
A set $\{\omega_{1},\nu_{1},\ldots,\omega_{l},\nu_{l}\}$, which will be called a block of type $\overline{B}_{l}^{\mathrm{I}}(u,\beta)$, or simply be called an odd block of type I, is given by a family of $2l$ independent
odd elements in $L$,
satisfying that the only nonzero products among them are the following:
          \begin{eqnarray*}
          && [u,\omega_{i}]=-\mathbf{i}\beta \omega_{i+1},\\
              && [u,\omega_{l}]=-\mathbf{i}\beta \omega_{1},\\
              && [u,\nu_{i}]=-\mathbf{i}\beta \nu_{i+1}, \\
              && [u,\nu_{l}]=-(-1)^{l}\mathbf{i}\beta \nu_{1}, \\
              && [\omega_{i},\nu_{l-i}]=(-1)^{l-i}z,\\
              && [\omega_{l},\nu_{l}]=(-1)^{l}z,
               \end{eqnarray*}
               where $1\leq i\leq l-1$.
\end{definition}

To give the next example, fix a primitive $2l$th root $\zeta$ of  unit and nonzero scalars $\alpha,\beta$. Consider now the twisted Heisenberg superalgebra
$$H^{\lambda,\kappa}_{2l+2,2l+2s+p}=\langle z,u,u_{1},v_{1},\ldots,u_{l},v_{l}\mid f_{1},g_{1},\ldots,f_{l},g_{l},p_{1},q_{1},\ldots,p_{s},q_{s},z_{1},\ldots,z_{p}\rangle$$ corresponding to $$\lambda=(\lambda_{1},\ldots,\lambda_{l})=(\zeta\alpha,\zeta^{2}\alpha,\ldots,\zeta^{l-1}\alpha,-\alpha),$$
$$\kappa=(\kappa_{1},\ldots,\kappa_{l})=(\zeta\beta,\zeta^{2}\beta,\ldots,\zeta^{l-1}\beta,-\beta),$$
and the the nonzero Lie brackets are given by
 \begin{eqnarray*}
   && [u,u_{i}]=-\mathbf{i}\zeta^{i}\alpha u_{i}, \\
   && [u,v_{i}]=\mathbf{i}\zeta^{i}\alpha v_{i}, \\
   && [u_{i},v_{i}]=-\zeta^{i}\alpha z, \\
   && [u,f_{i}]=-\mathbf{i}\zeta^{i}\beta f_{i},\\
   && [u,g_{i}]=\mathbf{i}\zeta^{i}\beta g_{i}, \\
   && [f_{i},g_{i}]=[p_{j},q_{j}]=[z_{t},z_{t}]=z,
\end{eqnarray*}
where $1\leq i\leq l$, $1\leq j\leq s$ and $1\leq t\leq p$.
Take now
\begin{eqnarray*}
   && x_{j}=\frac{\mathbf{i}}{\sqrt{2l}}\sum\limits_{i=1}^{l}(u_{i}+(-1)^{j-1}v_{i})\zeta^{(j-1)i} \\
   && y_{j}=\frac{\mathbf{i}}{\sqrt{2l}}\sum\limits_{i=1}^{l}(f_{i}+(-1)^{j-1}g_{i})\zeta^{ji}
\end{eqnarray*}
for each $1\leq j\leq 2l$. Observe that $\{x_{1},\ldots,x_{2l}\mid y_{1},\ldots,y_{2l}\}$ is a family of linearly independent elements satisfying
\begin{eqnarray*}
              && [u,x_{i}]=-\mathbf{i}\alpha x_{i+1},\\
              && [u,x_{2l}]=-\mathbf{i}\alpha x_{1},\\
              && [u,y_{i}]=-\mathbf{i}\beta y_{i+1},\\
              && [u,y_{2l}]=-\mathbf{i}\beta y_{1},
                      \end{eqnarray*}
                      where $1\leq i\leq 2l-1$.
A direct computation gives $$[x_{i},x_{j}]=\frac{1}{2l}\alpha((-1)^{i}+(-1)^{j-1})(\sum\limits_{k=1}^{l}\zeta^{(i+j-1)k})z$$
for all $1\leq i,j\leq 2l$.
If $i+j-1=2l,$ then $i$ and $j-1$ have the same parity and $[x_{i},x_{2l+1-i}]=(-1)^{i}\alpha z\neq 0$. Hence,
$$[x_{1},x_{2l}]=-[x_{2},x_{2l-1}]=\cdots=(-1)^{l-1}[x_{l},x_{l+1}]=-\alpha z$$ and the remaining brackets are zero:
if $r=i+j-1$ is odd, then $(-1)^{i}+(-1)^{j-1}=0$, and, if $r=i+j-1\neq 2l$ is even, then $\sum_{k=1}^{l}\zeta^{rk}=0$ since $\zeta^{2}$ is a primitive $l$th root of unit.

Similarly we have $$[y_{i},y_{j}]=\frac{1}{2l}((-1)^{i}+(-1)^{j})(\sum\limits_{k=1}^{l}\zeta^{(i+j)k})z$$
for all $1\leq i,j\leq 2l$. If $i+j=2l,$ then $i$ and $j$ have the same parity and $[y_{i},y_{2l-i}]=(-1)^{i}z\neq 0$. Hence,
$$-[y_{1},y_{2l-1}]=[y_{2},y_{2l-2}]=\cdots=(-1)^{l}[y_{l},y_{l}]= z$$ and the remaining brackets are zero.

We note that this provides a fine grading on $L=H_{2l+2,2l+2s+p}^{\lambda,\kappa}$ over the group $G=\mathbb{Z}\times \mathbb{Z}_{4l}\times \mathbb{Z}^{s}\times \mathbb{Z}_{2}^{p}$ given by
\begin{eqnarray*}
   &&  \langle u\rangle=L_{(0;\overline{2};0,\ldots,0;\overline{0},\ldots,\overline{0})},\\
   && \langle z\rangle=L_{(2;\overline{2};0,\ldots,0;\overline{0},\ldots,\overline{0})}, \\
   && \langle x_{i}\rangle=L_{(1;\overline{2i};0,\ldots,0;\overline{0},\ldots,\overline{0})},\\
   && \langle y_{i}\rangle=L_{(1;\overline{2i+1};0,\ldots,0;\overline{0},\ldots,\overline{0})},\\
   && \langle p_{i}\rangle=L_{(1;\overline{1};0,\ldots,1,\ldots,0;\overline{0},\ldots,\overline{0})} (1~ \hbox{in the}~ i\hbox{-th slot}),\\
   && \langle q_{i}\rangle=L_{(1;\overline{1};0,\ldots,-1,\ldots,0;\overline{0},\ldots,\overline{0})},\\
   && \langle z_{i}\rangle=L_{(1;\overline{1};0,\ldots,0;\overline{0},\ldots,\overline{1},\ldots,\overline{0})}.
\end{eqnarray*}
Clearly, there exists a unique $\rho\in \mathrm{Aut}(L)$ such that $$\rho(x_{i})=x_{l+i},~~~ \rho(y_{i})=y_{l+i},~~~ \rho(z)=(-1)^{l}z$$ for all $i=1,\ldots,2l ~(\mathrm{mod}~ 2l)$ and $\rho|_{V}=\mathbf{id}$, where $V$ is the subspace spanned by  $$\{u, p_{i},q_{i},z_{j}\mid i\leq s, j\leq p\}.$$ Take $\varrho_{i}$, $\widetilde{\tau}$ and $\widetilde{\varsigma}$ described in Theorem \ref{classification1}. Then the Weyl group of the grading is isomorphic to $$\mathbb{Z}_{2}^{1+s}\rtimes(S_{s}\times S_{p}),$$ since it is easily proved to be generated by $[\rho]$,  $[\varrho_{k_{i}}]$, $[\widetilde{\tau}]$ and $[\widetilde{\varsigma}]$.

This example motivates the following definition.
\begin{definition}
Let $L$ be a Lie  superalgebra, $z\in L$ a fixed even element, $u$ an arbitrary even element and $\alpha,\beta\in \mathbb{C}^{\times}$. A set $\{x_{1},x_{2},\ldots,x_{2l}\}$, which will be called a block of type $B_{l}^{\mathrm{II}}(u,\alpha)$, or simply be called an even block of type II, is given by a family of $2l$ independent
even elements in $L$,
satisfying the only nonzero products among them are the following:
\begin{eqnarray*}
              && [u,x_{i}]=-\mathbf{i}\alpha x_{i+1}(\mathrm{mod}~ 2l),\\
              && [x_{i},x_{2l-i+1}]=(-1)^{i}\alpha z,
              \end{eqnarray*}
              where $1\leq i\leq 2l.$
              A set $\{y_{1},y_{2},\ldots,y_{2l}\}$, which will be called a block of type $\overline{B}_{l}^{\mathrm{II}}(u,\alpha)$, or simply be called an odd block of type II, is given by a family of $2l$ independent
odd elements in $L$,
satisfying that the only nonzero products among them are the following:
\begin{eqnarray*}
              && [u,y_{i}]=-\mathbf{i}\beta y_{i+1}(\mathrm{mod}~ 2l),\\
              && [y_{i},y_{2l-i}]=(-1)^{i}z,
              \end{eqnarray*}
          where $1\leq i\leq 2l.$
               A set $\{p_{1},q_{1},\ldots,p_{s},q_{s},z_{1},\ldots,z_{p}\}$, which will be called a block of type $B_{s,p}^{\mathrm{III}}(u)$, or simply be called a block of type III, is given by a family of $2s+p$ independent
odd elements in $L$,
satisfying that the only nonzero products among them are the following:
          \begin{eqnarray*}
          && [p_{i},q_{i}]=z,\\
              && [z_{j},z_{j}]=z,
               \end{eqnarray*}
               where $1\leq i\leq s$ and $1\leq j\leq p$.
\end{definition}
In fact, all of the gradings on a twisted Heisenberg superalgebra can be described with blocks of type I, II and III, according to the following proposition.
\begin{proposition}\label{classification2}
Let $\Gamma$ be a $G$-grading on $L=H_{n,m}^{\lambda,\kappa}$. Let $z\in \mathrm{Z}(L)$. Then there exist $u\in L$, positive integers $l,t,t',q,q',s,p$ such that
$$l(q+q'+2t+2t')+2s+p=n+m-1$$ $(q,q'=0$ when $l$ is odd) and scalars $\beta_{1},\ldots,\beta_{t}$, $\beta_{1}',\ldots,\beta_{t'}'$, $\alpha_{1},\ldots,\alpha_{q}$, $\alpha_{1}',\ldots,\alpha_{q'}'\in \{\pm\lambda_{1},\ldots,\pm\lambda_{k},\pm \kappa_{1},\ldots,\pm \kappa_{r}\}$ such that
\begin{equation}\label{block}
  \{z,u\}\cup(\bigcup\limits_{j=1}^{t}X_{j})\cup(\bigcup\limits_{j=1}^{t'}X'_{j})
\cup(\bigcup\limits_{j=1}^{q}Y_{j})\cup(\bigcup\limits_{j=1}^{q'}Y'_{j})\cup T
\end{equation}
is a basis of homogeneous elements of $\Gamma$, where $X_{j}$ is a block of type $B_{l}^{\mathrm{I}}(u,\beta_{j})$ for all $j=1,\ldots,t$, $X'_{j}$ is a block of type $\overline{B}_{l}^{\mathrm{I}}(u,\beta_{j}')$ for all $j=1,\ldots,t'$, $Y_{j}$ is a block of type $B_{\frac{l}{2}}^{\mathrm{II}}(u,\alpha_{j})$ for all $j=1,\ldots,q$, $Y'_{j}$ is a block of type $\overline{B}_{\frac{l}{2}}^{\mathrm{II}}(u,\alpha_{j}')$ for all $j=1,\ldots,q'$, $T$ is a block of type $B_{s,p}^{\mathrm{III}}(u)$,  the brackets of any two elements belonging to different blocks are zero.

\end{proposition}
\begin{proof}
Recall that $z$ is always a homogeneous element. By Lemma \ref{homo basis}, we can assume that $u$ is also homogeneous of degree $h\in G$, necessarily of finite order. Let $l\in \mathbb{Z}_{\geq 0}$ be the order of $h$. Take $\varphi=\mathrm{ad}u$ and consider again the subspaces $V_{\mu_{i}}$ and $V_{\mu_{i}}^{l}$.
Recall that $V_{\mu_{i}}^{l}$ is a $\varphi$-invariant $G$-graded subspace for all $\mu_{i}\in \mathrm{Spec}(u)$.

(1) Let us discuss first the case that $l$ is odd. Fix any $0\neq x\in V_{\mu_{1}}^{l}\cap L_{g}\cap L_{\overline{0}}$ for some $g\in G.$ Since each $\varphi^{i}(x)\in L_{g+ih}$, one sees that $$\{x,\varphi(x),\ldots,\varphi^{l-1}(x)\}$$ is a family of linearly independent elements of $L$. Now observe that Eq. (\ref{eigenspace}), together with the fact that $l$ is odd,  one sees that $[V_{\mu_{1}}^{l}\cap L_{\overline{0}},V_{\mu_{1}}^{l}\cap L_{\overline{0}}]=0$ and $[V_{\mu_{1}}^{l}\cap L_{\overline{0}},V_{-\mu_{1}}^{l}\cap L_{\overline{0}}]\neq0$. From here,
$[\varphi^{i}(x),\varphi^{j}(x)]=0$ for any $i$, $j=0,1,\ldots,l-1$, and we can take a nonzero homogeneous element $0\neq y\in V_{-\mu_{1}}^{l}\cap L_{p}\cap L_{\overline{0}}$
such that $[x,y]\neq 0.$ By scaling if necessary, we can suppose $[x,y]=\lambda_{1}z$. Then $\mathrm{deg}z=g+p$. As above, we also have that $$\{y,\varphi(y),\ldots,\varphi^{l-1}(y)\}$$ is a family of linearly independent elements of $L$ satisfying $$[\varphi^{i}(y),\varphi^{j}(y)]=0$$
for any $i$, $j=0,1,\ldots,l-1$. Since $\varphi^{i}(x)\in L_{g+ih}$ and $\varphi^{j}(y)\in L_{p+jh}$, we get that if $[\varphi^{i}(x),\varphi^{j}(y)]\neq0$, then $g+p+(i+j)h=\mathrm{deg}z=g+p$. Hence we have $i+j$ is a multiple of $l$. That is, for each $0\leq i,j<l,$ $$[\varphi^{i}(x),\varphi^{j}(y)]=0$$ if $i+j\neq 0,l.$ Also note that $$[\varphi^{i}(x),\varphi^{l-i}(y)]=(-1)^{l-i}\mu_{1}^{l}[x,y]\neq0.$$
Indeed, take $\xi$ a primitive $l$th root of  unit and write $x=\sum_{j=0}^{l-1}a_{j}$ and $y=\sum_{j=0}^{l-1}b_{j}$ for $a_{j}\in V_{\xi^{j}\mu_{1}}\cap L_{\overline{0}}$ and
$b_{j}\in V_{-\xi^{j}\mu_{1}}\cap L_{\overline{0}}$. Then
\begin{eqnarray*}
  [\varphi^{i}(x),\varphi^{l-i}(y)] &=& \sum\limits_{j,k=0}^{l-1}[\xi^{ji}\mu_{1}^{i}a_{j},(-\xi^{k})^{l-i}\mu_{1}^{l-i}b_{k}] \\
   &=& \mu_{1}^{l}(-1)^{l-i}\sum\limits_{j=0}^{l-1}\xi^{ji}\xi^{j(l-i)}[a_{j},b_{j}]=(-1)^{l-i}\mu_{1}^{l}[x,y].
\end{eqnarray*}
It is clear that the family $$\{x,\varphi(x),\ldots,\varphi^{l-1}(x),y,\varphi(y),\ldots,\varphi^{l-1}(y)\}$$
is linearly independent. Recall that $\mu_{1}=-\mathbf{i}\lambda_{1}$. We have
$$W:=\left\{\frac{\varphi(x)}{\mu_{1}},\frac{\varphi(y)}{\mu_{1}},\cdots,\frac{\varphi^{i}(x)}{\mu_{1}^{i}},\frac{\varphi^{i}(y)}{\mu_{1}^{i}},\cdots,
\frac{\varphi^{l}(x)}{\mu_{1}^{l}}=x,\frac{\varphi^{l}(y)}{\mu_{1}^{l}}=(-1)^{l}y\right\}$$ is a block of type $B_{l}^{\mathrm{I}}(u,\lambda_{1})$.

Now $[u,L]=W\oplus \mathrm{Z}_{[u,L]}(W)$, where $W$ and its centralizer $\mathrm{Z}_{[u,L]}(W)$ are $G$-graded and $\varphi$-invariant. We continue  this process on $\mathrm{Z}_{[u,L]}(W)$ until finding all $t+t'$ blocks of type I of homogeneous elements.

(2) Now consider the case with $l$ even. Take a linear subspace $V_{\mu_{j}^{l}}\neq 0$ as above. Then we have two different cases to distinguish.

(i) Suppose $1\leq j\leq k.$

Case 1. Assume  for any $g\in G$ and  $x\in V_{\mu_{j}}^{l}\cap L_{g}\cap L_{\overline{0}}$ we have $[x,\varphi(x)]=0$. Fix $0\neq x\in V_{\mu_{j}}^{l}\cap L_{g}\cap L_{\overline{0}}$ for some $g\in G$. Then $$\{x,\varphi(x),\ldots,\varphi^{l-1}(x)\}$$ is a family of linearly independent elements of $L$. By induction on $n$ it is easy to show $[\varphi^{i}(x),\varphi^{i+n}(x)]=0$ for any $i=0,\ldots,l-1$ and $n=1,\ldots,l$.  That is, $[\varphi^{i}(x),\varphi^{i'}(x)]=0$ for any $i,i'=0,\ldots,l-1.$

Since the fact $l$ is even implies $V_{\mu_{j}}^{l}=V_{-\mu_{j}}^{l}$, we can choose a homogeneous element $0\neq y\in V_{\mu_{j}}^{l}\cap L_{g'}\cap L_{\overline{0}}$ for some $g'\in G$, such that $0\neq[x,y]=\lambda_{j}z.$ The same arguments that in the odd case say that again $$W:=\left\{\frac{\varphi(x)}{\mu_{j}},\frac{\varphi(y)}{\mu_{j}},\cdots,\frac{\varphi^{i}(x)}{\mu_{j}^{i}},\frac{\varphi^{i}(y)}{\mu_{j}^{i}},\cdots,
\frac{\varphi^{l}(x)}{\mu_{j}^{l}},\frac{\varphi^{l}(y)}{\mu_{j}^{l}}\right\}$$ is a block of type $B_{l}^{\mathrm{I}}(u,\lambda_{j})$. Now we can write
$$[u,L]=W\oplus \mathrm{Z}_{[u,L]}(W),$$ where $W$ and $\mathrm{Z}_{[u,L]}(W)$ are $G$-graded and $\varphi$-invariant.

Case 2. Assume that there exists a  nonzero  element  $x\in V_{\mu_{j}}^{l}\cap L_{g}\cap L_{\overline{0}}$ for some $g\in G$ such that $[x,\varphi(x)]\neq0.$  By scaling if necessary, we can assume $[x,\varphi(x)]=\mu_{j}\lambda_{j}z$. As above one sees that $$\{x,\varphi(x),\ldots,\varphi^{l-1}(x)\}$$ is a family of homogeneous linearly independent elements of $L$ satisfying $$[\varphi^{i}(x),\varphi^{i'}(x)]=0$$ if $i+i'\neq 1,l+1$.
A direct computation shows that $$[\varphi^{i}(x),\varphi^{l-i+1}(x)]=(-1)^{i}\mu_{j}^{l}[x,\varphi(x)]=(-1)^{i}\mu_{j}^{l+1}\lambda_{j}z$$
for any $i=1,\ldots,l.$ Thus the set $$W:=\left\{\frac{\varphi(x)}{\mu_{j}},\cdots,\frac{\varphi^{i}(x)}{\mu_{j}^{i}},\cdots,
\frac{\varphi^{l}(x)}{\mu_{j}^{l}}=x\right\} $$ is a block of type $B_{\frac{l}{2}}^{\mathrm{II}}(u,\lambda_{j})$.
Now we can write
$$[u,L]=W\oplus \mathrm{Z}_{[u,L]}(W),$$ where $W$ and $\mathrm{Z}_{[u,L]}(W)$ are $G$-graded and $\varphi$-invariant.

(ii) Suppose $k+1\leq j\leq k+r.$

Case 1. Assume  for any $g\in G$ and any $x\in V_{\mu_{j}}^{l}\cap L_{g}\cap L_{\overline{1}}$ we have $[x,x]=0$. Fix $0\neq x\in V_{\mu_{j}}^{l}\cap L_{g}\cap L_{\overline{1}}$ for some $g\in G$. Then $$\{x,\varphi(x),\ldots,\varphi^{l-1}(x)\}$$ is a family of linearly independent elements of $L$. A similar argument as in (i) shows that $[\varphi^{i}(x),\varphi^{i'}(x)]=0$ for any $i,i'=0,\ldots,l-1.$

Since the fact $l$ is even implies $V_{\mu_{j}}^{l}=V_{-\mu_{j}}^{l}$, we can choose a homogeneous element $0\neq y\in V_{\mu_{j}}^{l}\cap L_{g'}\cap L_{\overline{1}}$ for some $g'\in G$ such that $0\neq[x,y]=\kappa_{j-k}z.$ Then  $$\left\{\frac{\varphi(x)}{\mu_{j}},\frac{\varphi(y)}{\mu_{j}},\cdots,\frac{\varphi^{i}(x)}{\mu_{j}^{i}},\frac{\varphi^{i}(y)}{\mu_{j}^{i}},\cdots,
\frac{\varphi^{l}(x)}{\mu_{j}^{l}},\frac{\varphi^{l}(y)}{\mu_{j}^{l}}\right\}$$ is  a block of type $\overline{B}_{l}^{\mathrm{I}}(u,\kappa_{j-k})$. Now we can write
$$[u,L]=W\oplus \mathrm{Z}_{[u,L]}(W),$$where $W$ and $\mathrm{Z}_{[u,L]}(W)$ are $G$-graded and $\varphi$-invariant.

Case 2. Assume that there exists a nonzero even element  $x\in V_{\mu_{j}}^{l}\cap L_{g}\cap L_{\overline{1}}$ for some $g\in G$ such that $[x,x]\neq0.$  By scaling if necessary, we can assume $[x,x]=z$. We have  $$\{x,\varphi(x),\ldots,\varphi^{l-1}(x)\}$$ is a family of homogeneous linearly independent elements of $L$ satisfying $$[\varphi^{i}(x),\varphi^{i'}(x)]=0$$ if $i+i'\neq l,2l$.
By a direct computation we have $$[\varphi^{i}(x),\varphi^{l-i}(x)]=(-1)^{i}\mu_{j}^{l}[x,x]=(-1)^{i}\mu_{j}^{l}z$$
for any $i=1,\ldots,l.$ Thus the set $$\left\{\frac{\varphi(x)}{\mu_{j}},\cdots,\frac{\varphi^{i}(x)}{\mu_{j}^{i}},\cdots,
\frac{\varphi^{l}(x)}{\mu_{j}^{l}}=x\right\} $$ is a block of type $B_{\frac{l}{2}}^{\mathrm{II}}(u,\kappa_{j-k})$.

 In summary, we can continue  this process on $\mathrm{Z}_{[u,L]}(W)$ until finding all $t+t'$ blocks of type I and $q+q'$ blocks of type II.

(3) Finally, we will find the block of type III. By \cite[Lemma 2]{Cal2}, for each grading $\Gamma$, there is a homogeneous basis
$B=\{p_{1},q_{1},\ldots,p_{s},q_{s},z_{1},\ldots,z_{p}\}$ of $V:=\ker \varphi\cap L_{\overline{1}}$ such that
\begin{eqnarray*}
          && [p_{i},q_{i}]=z,~~~1\leq i\leq s,\\
              && [z_{i},z_{i}]=z,~~~1\leq i\leq p.
               \end{eqnarray*}
               Thus the set $\{p_{1},q_{1},\ldots,p_{s},q_{s},z_{1},\ldots,z_{p}\}$ is a block of type $B_{s,p}^{\mathrm{III}}(u)$.
\end{proof}
Proposition \ref{classification2} provides  a description of all the fine grading on $H_{n,m}^{\lambda,\kappa}$. It is clear that each basis as the one in Eq. (\ref{block}) determines a fine grading, which is denoted by
\begin{equation}\label{block grading}
  \Gamma(l,t,t',q,q',s,p;\beta,\beta',\alpha,\alpha'),
\end{equation}
where $\beta=(\beta_{1},\ldots,\beta_{t})$, $\beta'=(\beta_{1}',\ldots,\beta_{t'}')$, $\alpha=(\alpha_{1},\ldots,\alpha_{q})$ and $\alpha'=(\alpha_{1}',\ldots,\alpha_{q'}').$
If we are in the situation of Proposition \ref{classification2} and take $\xi$ a primitive $l$th root of unit, by Eq. (\ref{spec}), then
\begin{equation}\label{specu}
\begin{split}
\mathrm{Spec}(u)
   &= \{\pm\mathbf{i}\xi^{c}\beta_{j},\pm\mathbf{i}\xi^{c}\beta_{j'}'\mid j=1,\ldots,t, j'=1,\ldots,t',c=1,\ldots,l \}\\
   &\cup \{\pm\mathbf{i}\xi^{d}\alpha_{j},\pm\mathbf{i}\xi^{d}\alpha_{j'}'\mid j=1,\ldots,q, j'=1,\ldots,q', d=1,\ldots,\frac{l}{2}\}.
\end{split}
\end{equation}
Take
\begin{eqnarray*}
   && Y_{i}:=\{\sqrt{\mathbf{i}}\widehat{e_{i}},\frac{e_{i}}{\sqrt{\mathbf{i}}}\},~~~~~~~ Y'_{i}:=\{\mathbf{i}w_{i},-\widehat{w_{i}}\},\\
   && X_{i}:=\{u_{i},v_{i}\},~~~~~~~~~~~ X'_{i}:=\{\mathbf{i}f_{i},\mathbf{i}g_{i}\}.
\end{eqnarray*}
Then $X_{i}$ is a block of type $B_{1}^{\mathrm{I}}(u,\lambda_{i})$ for all $i=1,\ldots,k$, $X'_{i}$ is a block of type $\overline{B}_{1}^{\mathrm{I}}(u,\kappa_{i})$ for all $i=1,\ldots,r$, $Y_{i}$ is a block of type $B_{1}^{\mathrm{II}}(u,\lambda_{i})$ for all $i=1,\ldots,k$, $Y'_{i}$ is a block of type $\overline{B}_{1}^{\mathrm{II}}(u,\kappa_{i})$ for all $i=1,\ldots,r$.
So we have our familiar fine gradings of Section \ref{section first grading} are
\begin{eqnarray*}
   && \Gamma_{1}^{s}=\Gamma(2,0,0,k,r,s,m-2r-2s;\lambda, \kappa), \\
   && \Gamma_{2}^{s}=\Gamma(1,k,r,0,0,s,m-2r-2s;\lambda, \kappa),
\end{eqnarray*}
where $\lambda=(\lambda_{1},\ldots,\lambda_{k})$ and $\kappa=(\kappa_{1},\ldots,\kappa_{r})$.

 Next we will classify all the fine gradings up to equivalence. Similar to \cite[Lemma 5]{Cal2}, we have the following lemma.

\begin{lemma}\label{same block}
Let $L$ be a Lie superalgebra, $z\in \mathrm{Z}(L)$, $u\in L$. Suppose $L$ contains some blocks of types I or II.

  $(\mathrm{1})$ Let $X\subset L$ be a block of type $B_{l}^{\mathrm{I}}(u,\alpha)$ (or $\overline{B}_{l}^{\mathrm{I}}(u,\alpha))$ and $Y\subset L$ be a block of type $B_{l}^{\mathrm{I}}(u,\beta)$ (or $\overline{B}_{l}^{\mathrm{I}}(u,\beta))$, where  $\alpha, \beta\in \mathbb{C}^{\times}$. Then $\langle X\rangle=\langle Y\rangle$  if and only if $(\frac{\alpha}{\beta})^{l}=1$ if $l$
is even and $(\frac{\alpha}{\beta})^{2l}=1$ if $l$ is odd.

$(\mathrm{2})$ Let $X'\subset L$ be a block of type $B_{l}^{\mathrm{II}}(u,\alpha')$ (or $\overline{B}_{l}^{\mathrm{II}}(u,\alpha')$) and $Y'\subset L$ be a block of type $B_{l}^{\mathrm{II}}(u,\beta')$ (or $B_{l}^{\mathrm{II}}(u,\beta')$),  where $\alpha', \beta'\in \mathbb{C}^{\times}$. Then $\langle X'\rangle=\langle Y'\rangle$  if and only if $(\frac{\alpha'}{\beta'})^{l}=1$.
\end{lemma}
From now on $\xi,\zeta\in \mathbb{C}$ will denote some fixed $l$th and $2l$th primitive roots of unit, respectively.
Let
\begin{eqnarray*}
   && \widehat{\alpha}:=\{\alpha\xi^{t}: t=0,\ldots,l-1\}, \\
   && \widetilde{\alpha}:=\{\alpha\zeta^{t}: t=0,\ldots,2l-1\},
\end{eqnarray*}
where $\alpha\in \mathbb{C}$.

Now we are able to classify all the fine gradings up to equivalencefor twisted Heisenberg superalgebras.

\begin{theorem}\label{equivalent}
$(\mathrm{1})$ A grading $\Gamma$ on $H_{n,m}^{\lambda,\kappa}$ is fine if and only if $\Gamma$ is equivalent to $$\Gamma(l,t,t',q,q',s,p;\beta,\beta',\alpha,\alpha')$$
for some $l,t,t',q,q',s,p\in \mathbb{Z}_{\geq0}$ such that
$$l(q+q'+2t+2t')+2s+p=n+m-1$$ and for some $\beta=(\beta_{1},\ldots,\beta_{t})$, $\beta'=(\beta_{1}',\ldots,\beta_{t'}')$, $\alpha=(\alpha_{1},\ldots,\alpha_{q})$, $\alpha'=(\alpha_{1}',\ldots,\alpha_{q'}')$  such that Eq. (\ref{specu}) holds, with $l$ even
if $q\neq 0$ or $q'\neq 0.$

$(\mathrm{2})$ If $l$ is odd, two fine gradings on $H_{n,m}^{\lambda,\kappa}$, $$\Gamma=\Gamma(l,t,t',q,q',s,p;\beta,\beta',\alpha,\alpha'),$$ where $\beta=(\beta_{1},\ldots,\beta_{t})$, $\beta'=(\beta_{1}',\ldots,\beta_{t'}')$, $\alpha=(\alpha_{1},\ldots,\alpha_{q})$, $\alpha'=(\alpha_{1}',\ldots,\alpha_{q'}')$ and $$\overline{\Gamma}=\Gamma(\overline{l},\overline{t},\overline{t}',\overline{q},\overline{q}',\overline{s},\overline{p};
\overline{\beta},\overline{\beta}',
\overline{\alpha},\overline{\alpha}'),$$ where $\overline{\beta}=(\overline{\beta}_{1},\ldots,\overline{\beta}_{\overline{t}})$, $\overline{\beta}'=(\overline{\beta}_{1}',\ldots,\overline{\beta}_{\overline{t}'}')$, $\overline{\alpha}=(\overline{\alpha}_{1},\ldots,\overline{\alpha}_{\overline{q}})$, $\overline{\alpha}'=(\overline{\alpha}_{1}',\ldots,\overline{\alpha}_{\overline{q}'}')$,
are equivalent if and only if $l=\overline{l}$, $t=\overline{t}$, $t'=\overline{t}'$, $q=\overline{q}$, $q'=\overline{q}'$, $s=\overline{s}$, $p=\overline{p}$
and there are $\varepsilon\in \mathbb{C}^{\times}$, $\eta\in S_{t}$, $\eta'\in S_{t'}$, $\sigma\in S_{q}$, $\sigma'\in S_{q'}$ such that for all $j=1,\ldots,t$,
$j'=1,\ldots,t'$, $i=1,\ldots,q$, $i'=1,\ldots,q'$,
$$\widehat{\varepsilon\overline{\beta}_{j}}=\widehat{\beta_{\eta(j)}},~~~
\widehat{\varepsilon\overline{\beta}_{j'}'}=\widehat{\beta'_{\eta'(j')}},~~~
  \widehat{\varepsilon\overline{\alpha}_{i}}=\widehat{\alpha_{\sigma(i)}},~~~
  \widehat{\varepsilon\overline{\alpha}'_{i'}}=\widehat{\alpha'_{\sigma'(i')}}.$$

$(\mathrm{3})$ If $l$ is even, two fine gradings on $H_{n,m}^{\lambda,\kappa}$, $$\Gamma=\Gamma(l,t,t',q,q',s,p;\beta,\beta',\alpha,\alpha'),$$ where $\beta=(\beta_{1},\ldots,\beta_{t})$, $\beta'=(\beta_{1}',\ldots,\beta_{t'}')$, $\alpha=(\alpha_{1},\ldots,\alpha_{q})$, $\alpha'=(\alpha_{1}',\ldots,\alpha_{q'}')$ and $$\overline{\Gamma}=\Gamma(\overline{l},\overline{t},\overline{t}',\overline{q},\overline{q}',\overline{s},\overline{p};
\overline{\beta},\overline{\beta}',
\overline{\alpha},\overline{\alpha}'),$$ where $\overline{\beta}=(\overline{\beta}_{1},\ldots,\overline{\beta}_{\overline{t}})$, $\overline{\beta}'=(\overline{\beta}_{1}',\ldots,\overline{\beta}_{\overline{t}'}')$, $\overline{\alpha}=(\overline{\alpha}_{1},\ldots,\overline{\alpha}_{\overline{q}})$, $\overline{\alpha}'=(\overline{\alpha}_{1}',\ldots,\overline{\alpha}_{\overline{q}'}')$,
are equivalent if and only if $l=\overline{l}$, $t=\overline{t}$, $t'=\overline{t}'$, $q=\overline{q}$, $q'=\overline{q}'$, $s=\overline{s}$, $p=\overline{p}$
and there are $\varepsilon\in \mathbb{C}^{\times}$, $\eta\in S_{t}$, $\eta'\in S_{t'}$, $\sigma\in S_{q}$, $\sigma'\in S_{q'}$ such that for all $j=1,\ldots,t$,
$j'=1,\ldots,t'$, $i=1,\ldots,q$, $i=1,\ldots,q'$,
$$\widetilde{\varepsilon\overline{\beta}_{j}}=\widetilde{\beta_{\eta(j)}},~~~
\widetilde{\varepsilon\overline{\beta}_{j'}'}=\widetilde{\beta_{\eta(j')}'},~~~
\widehat{\varepsilon\overline{\alpha}_{i}}=\widehat{\alpha_{\sigma(i)}},~~~
\widehat{\varepsilon\overline{\alpha}'_{i'}}=\widehat{\alpha'_{\sigma'(i')}}.$$
\end{theorem}
\begin{proof}
We only give the proof of (1). Similar to \cite[Theorem 5]{Cal2}, by Lemma \ref{same block}, we have (2) and (3).

 The fact that any fine grading is like Eq. (\ref{block grading}) has been proved in Proposition \ref{classification2}. For the converse, recorder
\begin{eqnarray*}
   && \lambda=(\xi\beta_{1},\ldots,\xi^{l}\beta_{1},\ldots,\xi\beta_{t},\ldots,\xi^{l}\beta_{t},\xi\alpha_{1},\ldots,\xi^{\frac{l}{2}}\alpha_{1},\ldots,
   \xi\alpha_{q},\ldots,\xi^{\frac{l}{2}}\alpha_{q}), \\
   && \kappa= (\xi\beta_{1}',\ldots,\xi^{l}\beta_{1}',\ldots,\xi\beta_{t'}',\ldots,\xi^{l}\beta_{t'}',\xi\alpha_{1}',\ldots,\xi^{\frac{l}{2}}\alpha_{1}',\ldots,
   \xi\alpha_{q'}',\ldots,\xi^{\frac{l}{2}}\alpha_{q'}'),
\end{eqnarray*}
which are possible because of Eq. (\ref{specu}).  Take $\{x_{1}^{j},y_{1}^{j},\ldots,x_{l}^{j},y_{l}^{j}\}$  to be a block of type $B_{l}^{\mathrm{I}}(u,\beta_{j})$ for each $j\leq t$ and $\{\overline{x}_{1}^{j},\overline{y}_{1}^{j},\ldots,\overline{x}_{l}^{j},\overline{y}_{l}^{j}\}$ a block of type $\overline{B}_{l}^{\mathrm{I}}(u,\beta_{j}')$ for each $j\leq t'$. Take $\{a_{1}^{j},\ldots,a_{l}^{j}\}$ to be a block of type $B_{\frac{l}{2}}^{\mathrm{II}}(u,\alpha_{j})$ for each $j\leq q$ and $\{\overline{a}_{1}^{j},\ldots,\overline{a}_{l}^{j}\}$ a block of type $\overline{B}_{\frac{l}{2}}^{\mathrm{II}}(u,\alpha_{j}')$ for each $j\leq q'$. Take
 $\{p_{1},q_{1},\ldots,p_{s},q_{s},z_{1},\ldots,z_{m-2r-2s}\}$ to be a block of type $B_{s,p}^{\mathrm{III}}(u)$, where $p=m-2r-2s.$

If $l$ is even,  the union of these blocks  determines a fine $G$-grading, which is given by
\begin{equation}\label{universal grading group1}
\begin{split}
   & \mathrm{deg}x_{i}^{j}=(1;\overline{2i+2};0,\ldots,1,\ldots,0;0,\ldots,0;\overline{0},\ldots,\overline{0};\overline{0},\ldots,\overline{0};0,\ldots,0;\overline{0},\ldots,\overline{0}),\\
   & \mathrm{deg}y_{i}^{j}=(1;\overline{2i};0,\ldots,-1,\ldots,0;0,\ldots,0;\overline{0},\ldots,\overline{0};\overline{0},\ldots,\overline{0};0,\ldots,0;\overline{0},\ldots,\overline{0}),  \\
   & \mathrm{deg}\overline{x}_{i}^{j}=(1;\overline{2i+2};0,\ldots,0;0,\ldots,1,\ldots,0;\overline{0},\ldots,\overline{0};\overline{0},\ldots,\overline{0};0,\ldots,0;\overline{0},\ldots,\overline{0}),  \\
   & \mathrm{deg}\overline{y}_{i}^{j}=(1;\overline{2i};0,\ldots,0;0,\ldots,-1,\ldots,0;\overline{0},\ldots,\overline{0};\overline{0},\ldots,\overline{0};0,\ldots,0;\overline{0},\ldots,\overline{0}),  \\
   & \mathrm{deg}a_{i}^{j}=(1;\overline{2i};0,\ldots,0;0,\ldots,0;\overline{0},\ldots,\overline{1},\ldots,\overline{0};\overline{0},\ldots,\overline{0};0,\ldots,0;\overline{0},\ldots,\overline{0}),  \\
   & \mathrm{deg}\overline{a}_{i}^{j}=(1;\overline{2i+1};0,\ldots,0;0,\ldots,0;\overline{0},\ldots,\overline{0};\overline{0},\ldots,\overline{1},\ldots,\overline{0};0,\ldots,0;\overline{0},\ldots,\overline{0}),  \\
   & \mathrm{deg}p_{i}=(1;\overline{1};0,\ldots,0;0,\ldots,0;\overline{0},\ldots,\overline{0};\overline{0},\ldots,\overline{0};0,\ldots,1,\ldots,0;\overline{0},\ldots,\overline{0}),  \\
   & \mathrm{deg}q_{i}=(1;\overline{1};0,\ldots,0;0,\ldots,0;\overline{0},\ldots,\overline{0};\overline{0},\ldots,\overline{0};0,\ldots,-1,\ldots,0;\overline{0},\ldots,\overline{0}),  \\
   & \mathrm{deg}z_{i}=(1;\overline{1};0,\ldots,0;0,\ldots,0;\overline{0},\ldots,\overline{0};\overline{0},\ldots,\overline{0};0,\ldots,0;\overline{0},\ldots,\overline{1},\ldots,\overline{0}),  \\
   & \mathrm{deg}u=(0;\overline{2};0,\ldots,0;0,\ldots,0;\overline{0},\ldots,\overline{0};\overline{0},\ldots,\overline{0};0,\ldots,0;\overline{0},\ldots,\overline{0}),  \\
   &  \mathrm{deg}z=(2;\overline{2};0,\ldots,0;0,\ldots,0;\overline{0},\ldots,\overline{0};\overline{0},\ldots,\overline{0};0,\ldots,0;\overline{0},\ldots,\overline{0}),
\end{split}
\end{equation}
where $G=\mathbb{Z}\times \mathbb{Z}_{2l}\times \mathbb{Z}^{t+t'}\times \mathbb{Z}_{2}^{q+q'}\times \mathbb{Z}^{s}\times \mathbb{Z}_{2}^{p}$.

If $l$ is odd,
there are no blocks of type II. Now we have a fine $G'$-grading, which is given by
\begin{equation}\label{universal grading group2}
\begin{split}
   & \mathrm{deg}x_{i}^{j}=(1;\overline{i+2};0,\ldots,1,\ldots,0;0,\ldots,0;0,\ldots,0;\overline{0},\ldots,\overline{0}),\\
   & \mathrm{deg}y_{i}^{j}=(1;\overline{i};0,\ldots,-1,\ldots,0;0,\ldots,0;0,\ldots,0;\overline{0},\ldots,\overline{0}),  \\
   & \mathrm{deg}\overline{x}_{i}^{j}=(1;\overline{i+2};0,\ldots,0;0,\ldots,1,\ldots,0;0,\ldots,0;\overline{0},\ldots,\overline{0}),  \\
   & \mathrm{deg}\overline{y}_{i}^{j}=(1;\overline{i};0,\ldots,0;0,\ldots,-1,\ldots,0;0,\ldots,0;\overline{0},\ldots,\overline{0}),  \\
   & \mathrm{deg}p_{i}=(1;\overline{1};0,\ldots,0;0,\ldots,0;0,\ldots,1,\ldots,0;\overline{0},\ldots,\overline{0}),  \\
   & \mathrm{deg}q_{i}=(1;\overline{1};0,\ldots,0;0,\ldots,0;0,\ldots,-1,\ldots,0;\overline{0},\ldots,\overline{0}),  \\
   & \mathrm{deg}z_{i}=(1;\overline{1};0,\ldots,0;0,\ldots,0;0,\ldots,0;\overline{0},\ldots,\overline{1},\ldots,\overline{0}),  \\
   & \mathrm{deg}u=(0;\overline{1};0,\ldots,0;0,\ldots,0;0,\ldots,0;\overline{0},\ldots,\overline{0}),  \\
   &  \mathrm{deg}z=(2;\overline{2};0,\ldots,0;0,\ldots,0;0,\ldots,0;\overline{0},\ldots,\overline{0}),
\end{split}
\end{equation}
where $G'=\mathbb{Z}\times \mathbb{Z}_{l}\times \mathbb{Z}^{t+t'}\times \mathbb{Z}^{s}\times \mathbb{Z}_{2}^{p}$.
\end{proof}
By Eq. (\ref{universal grading group1}) and Eq. (\ref{universal grading group2}), we can get the universal grading group of any fine grading on $H_{n,m}^{\lambda,\kappa}$.
\begin{corollary}
The universal grading group of any fine grading $$\Gamma=\Gamma(l,t,t',q,q',s,p;\beta,\beta',\alpha,\alpha')$$  on $H_{n,m}^{\lambda,\kappa}$, is
\begin{eqnarray*}
  \mathbb{Z}^{1+t+t'+s}\times\mathbb{Z}_{2l}\times \mathbb{Z}_{2}^{q+q'+p} &&\hbox{if l is even}, \\
  \mathbb{Z}^{1+t+t'+s}\times\mathbb{Z}_{l}\times \mathbb{Z}_{2}^{p} &&\hbox{if l is odd}.
\end{eqnarray*}
\end{corollary}

By Theorem \ref{equivalent}, when one wants to know how many gradings are on a particular twisted Heisenberg superalgebras $H_{n,m}^{\lambda,\kappa}$, it is enough to see how many ways are of splitting $$\{\pm\lambda_{1},\ldots,\pm\lambda_{k},\pm \kappa_{1},\ldots,\pm \kappa_{r}\}$$ in the way described in Eq. (\ref{specu}).
\begin{example}
Let us compute how many fine grading are there on $L=H_{6,4}^{(1,1;\mathbf{i})}$. As $l$ must divide $4$ and $1$ or $-1$ is a primitive $l$th root of  unit (recall Eq. (\ref{spec})), the possibilities are
\begin{itemize}
  \item $l=1$ with $(t,t',q,q',s,p)=(2,1,0,0,1,0)$ or $(2,1,0,0,0,2)$
  \item $l=2$ with $(t,t',q,q',s,p)=(1,0,0,1,1,0)$, $(1,0,0,1,0,2)$, $(0,0,2,1,1,0)$ or $(0,0,2,1,0,2)$.
\end{itemize}
So we have six fine gradings:
\begin{itemize}
  \item  $\mathbb{Z}^{5}$-grading $\Gamma(1,2,1,0,0,1,0;1,1,\mathbf{i})$ (the only toral fine grading, $\Gamma_{2}^{1}$);
  \item  $\mathbb{Z}^{4}\times \mathbb{Z}_{2}^{2}$-grading $\Gamma(1,2,1,0,0,0,2;1,1,\mathbf{i})$;
  \item  $\mathbb{Z}^{3}\times\mathbb{Z}_{4}\times \mathbb{Z}_{2}$-grading $\Gamma(2,1,0,0,1,1,0;1,\mathbf{i})$;
  \item  $\mathbb{Z}^{2}\times\mathbb{Z}_{4}\times \mathbb{Z}_{2}^{3}$-grading $\Gamma(2,1,0,0,1,0,2;1,\mathbf{i})$;
  \item  $\mathbb{Z}^{2}\times\mathbb{Z}_{4}\times \mathbb{Z}_{2}^{3}$-grading $\Gamma(2,0,0,2,1,1,0;1,1,\mathbf{i})$;
  \item  $\mathbb{Z}\times\mathbb{Z}_{4}\times \mathbb{Z}_{2}^{5}$-grading $\Gamma(2,0,0,2,1,0,2;1,1,\mathbf{i})$.
\end{itemize}
\end{example}

\subsection{Weyl groups in the general case}\label{sec.4.3}
Finally, we would like to compute the Weyl groups of the fine gradings on the twisted Heisenberg superalgebra $L=H_{n,m}^{\lambda,\kappa}$.

Let $\Gamma$, $\Gamma'$ be two fine gradings on $L$. Since $\mathcal{W}(\Gamma)=\mathcal{W}(\Gamma')$ if $\Gamma$ and $\Gamma'$ are equivalent, by Theorem \ref{equivalent}, without loss of generality, we can assume each fine grading $$\Gamma=\Gamma(l,t,t',q,q',s,p;\beta,\beta',\alpha,\alpha')$$ on $L$ satisfies
\begin{equation}\label{scalar chosen}
\begin{split}
&\beta=(\beta_{1},\ldots,\beta_{t})=(\delta_{1},\ldots,\delta_{1},\ldots,\delta_{\overline{t}},\ldots,\delta_{\overline{t}}),~ \hbox{each}~ \delta_{i} ~\hbox{repeated}~ m_{i}~\hbox{times},\\
&\beta'=(\beta_{1}',\ldots,\beta_{t'}')=(\delta_{1}',\ldots,\delta_{1}',\ldots,\delta_{\overline{t}'}',\ldots,\delta_{\overline{t}'}'),~ \hbox{each}~ \delta_{i}' ~\hbox{repeated}~ m_{i}'~\hbox{times},\\
&\alpha=(\alpha_{1},\ldots,\alpha_{q})=(\gamma_{1},\ldots,\gamma_{1},\ldots,\gamma_{\overline{q}},\ldots,\gamma_{\overline{q}}),~ \hbox{each}~ \gamma_{i} ~\hbox{repeated}~ n_{i}~\hbox{times},\\
&\alpha'=(\alpha_{1}',\ldots,\alpha_{q'}')=(\gamma_{1}',\ldots,\gamma_{1}',\ldots,\gamma_{\overline{q}'}',\ldots,\gamma_{\overline{q}'}'),~ \hbox{each}~ \gamma_{i}' ~\hbox{repeated}~ n_{i}'~\hbox{times},
\end{split}
\end{equation}
and for all $i\neq j$,
\begin{eqnarray}
    &&\widehat{\delta_{i}}\neq \widehat{\delta_{j}},~ \widehat{\delta_{i}'}\neq \widehat{\delta_{j}'}, ~\widehat{\gamma_{i}}\neq \widehat{\gamma_{j}}, ~ \widehat{\gamma_{i}'}\neq \widehat{\gamma_{j}'}~~~~ \hbox{if}~ l~\hbox{is even}, \\
   && \widetilde{\delta_{i}}\neq \widetilde{\delta_{j}},~ \widetilde{\delta_{i}'}\neq \widetilde{\delta_{j}'}~~~~~~~~~~~~~~~~~~~~~~~~~~~~ \hbox{if}~ l~\hbox{is odd} ~(q=q'=0).
\end{eqnarray}
Write $$\mathbf{m}:=(m_{1},\ldots,m_{\overline{t}}),~~\mathbf{m}':=(m_{1}',\ldots,m_{\overline{t}'}'),~~\mathbf{n}:=(n_{1},\ldots,n_{\overline{q}}),~~ \mathbf{n}':=(n_{1}',\ldots,\overline{n}'_{\overline{q}'}).$$
From now on, this grading is denoted by $$\Gamma=\Gamma(l,t,t',q,q',s,p;\mathbf{m},\mathbf{m}',\mathbf{n},\mathbf{n}';\beta,\beta',\alpha,\alpha').$$
A basis of homogeneous elements of this grading is formed by $z\in \mathrm{Z}(L)$, $u$, blocks $$\{x_{1}^{j},y_{1}^{j},\ldots,x_{l}^{j},y_{l}^{j}\}$$ of type $B_{l}^{\mathrm{I}}(u,\beta_{j})$
for $j\leq t$, blocks $$\{\overline{x}_{1}^{j},\overline{y}_{1}^{j},\ldots,\overline{x}_{l}^{j},\overline{y}_{l}^{j}\}$$ of type $\overline{B}_{l}^{\mathrm{I}}(u,\beta_{j}')$ for $j\leq t'$,
blocks $$\{a_{1}^{j},\ldots,a_{l}^{j}\}$$ of type $B_{\frac{l}{2}}^{\mathrm{II}}(u,\alpha_{j})$
 for $j\leq q$, blocks $$\{\overline{a}_{1}^{j},\ldots,\overline{a}_{l}^{j}\}$$ of type $\overline{B}_{\frac{l}{2}}^{\mathrm{II}}(u,\alpha_{j}')$  for $j\leq q'$, and the block $$\{p_{1},q_{1},\ldots,p_{s},q_{s},z_{1},\ldots,z_{p}\}$$ of type $B_{s,p}^{\mathrm{III}}(u)$.

For each $\sigma\in S_{n_{1}}\times\cdots\times S_{n_{\overline{q}}}$, $\sigma'\in S_{n_{1}'}\times\cdots\times S_{n_{\overline{q}'}'}$,  $\eta\in S_{m_{1}}\times\cdots\times S_{m_{\overline{t}}}$ and $\eta'\in S_{m_{1}'}\times\cdots\times S_{m_{\overline{t}'}'}$,
there exists a unique $\Upsilon_{(\eta,\eta',\sigma,\sigma')}\in \mathrm{Aut}(\Gamma)$ defined by $\Upsilon_{(\eta,\eta',\sigma,\sigma')}|_{V}=\mathbf{id}$, where $V$ is the subspace spanned by $$\{z, u, p_{1},q_{1},\ldots,p_{s},q_{s},z_{1},\ldots,z_{p}\},$$ and
$$x_{i}^{j}\mapsto x_{i}^{\eta(j)},~y_{i}^{j}\mapsto y_{i}^{\eta(j)}~,\overline{x}_{i}^{j}\mapsto \overline{x}_{i}^{\eta'(j)},~\overline{y}_{i}^{j}\mapsto \overline{y}_{i}^{\eta'(j)}~,a_{i}^{j}\mapsto a_{i}^{\sigma(j)},~\overline{a}_{i}^{j}\mapsto \overline{a}_{i}^{\sigma'(j)}$$
for all $i\leq l$,
which obviously preserves the grading but interchanges the blocks.

For each $j\leq t$, there exists a unique $\theta_{j}\in \mathrm{Aut}(\Gamma)$ such that $\theta_{j}|_{W}=\mathbf{id}$, where $W$ is the subspace spanned by $$ \{z, u, x_{i}^{c},y_{i}^{c},\overline{x}_{i}^{c'},\overline{y}_{i}^{c'},a_{i}^{d},\overline{a}_{i}^{d'},p_{i'},q_{i'},z_{j'}\},$$
where $i\leq l$, $i'\leq s$, $j'\leq p$, $d\leq q$, $d'\leq q'$, $c\leq t$, $c'\leq t'$, $c\neq j$,
and
$$\theta_{j}(x_{i}^{j})=\mathbf{i}x_{i+1}^{j},~~~~~~\theta_{j}(y_{i}^{j})=\mathbf{i}y_{i-1}^{j}$$
for all $i\leq l$ (indices taken modulo $l$).

For each $j\leq t'$, there exists a unique $\overline{\theta}_{j}\in \mathrm{Aut}(\Gamma)$ such that $\overline{\theta}_{j}|_{W'}$, where $W'$ is the subspace spanned by $$\{z, u, x_{i}^{c},y_{i}^{c},\overline{x}_{i}^{c'},\overline{y}_{i}^{c'},a_{i}^{d},\overline{a}_{i}^{d'},p_{i'},q_{i'},z_{j'}\},$$
where $i\leq l$, $i'\leq s$, $j'\leq p$, $d\leq q$, $d'\leq q'$, $c\leq t$, $c'\leq t'$, $c'\neq j$,
 and
$$\overline{\theta}_{j}(\overline{x}_{i}^{j})
=\mathbf{i}\overline{x}_{i+1}^{j},~~~~~~\overline{\theta}_{j}(\overline{y}_{i}^{j})=\mathbf{i}\overline{y}_{i-1}^{j}$$
for all $i\leq l$ (indices taken modulo $l$).

When $l$ is even, for each $j\leq t$, there exists a unique $\vartheta_{j}\in \mathrm{Aut}(\Gamma)$ such that $\vartheta_{j}|_{V}=\mathbf{id}$, where $V$ is the subspace spanned by $$\{z, u, x_{i}^{c},y_{i}^{c},\overline{x}_{i}^{c'},\overline{y}_{i}^{c'},a_{i}^{d},\overline{a}_{i}^{d'},p_{i'},q_{i'},z_{j'}\}$$ for $i\leq l$, $i'\leq s$, $j'\leq p$, $d\leq q$, $d'\leq q'$, $c\leq t$, $c'\leq t'$, $c\neq j$,  and  $$\vartheta_{j}(x_{i}^{j})=y_{i}^{j},~~~~~~\vartheta_{j}(y_{i}^{j})=-x_{i}^{j}$$ for all $i\leq l$;
for each $j\leq t'$, there exists a unique $\overline{\vartheta}_{j}\in \mathrm{Aut}(\Gamma)$ such that $\overline{\vartheta}_{j}|_{V'}=\mathbf{id}$, where  $V'$ is the subspace spanned by $$\{z, u, x_{i}^{c},y_{i}^{c},\overline{x}_{i}^{c'},\overline{y}_{i}^{c'},a_{i}^{d},\overline{a}_{i}^{d'},p_{i'},q_{i'},z_{j'}\}$$ for $i\leq l$, $i'\leq s$, $j'\leq p$, $d\leq q$, $d'\leq q'$, $c\leq t$, $c'\leq t'$, $c'\neq j$, and such that $$\overline{\vartheta}_{j}(\overline{x}_{i}^{j})=\overline{y}_{i}^{j},~~~~~~\overline{\vartheta}_{j}(\overline{y}_{i}^{j})=\overline{x}_{i}^{j}$$
for all $i\leq l$.

Also when $l$ is even, for each $j\leq q$, there exists a unique $\rho_{j}\in \mathrm{Aut}(\Gamma)$ such that $\rho_{j}|_{W}=\mathbf{id}$, where $W$ is the subspace spanned by $$\{z, u, x_{i}^{c},y_{i}^{c},\overline{x}_{i}^{c'},\overline{y}_{i}^{c'},a_{i}^{d},\overline{a}_{i}^{d'},p_{i'},q_{i'},z_{j'}\},$$ where $i\leq l$, $i'\leq s$, $j'\leq p$, $d\leq q$, $d'\leq q'$, $c\leq t$, $c'\leq t'$, $d\neq j$, and such that $$\rho_{j}(a_{i}^{j})=a_{\frac{l}{2}+i}^{j}$$ for all $i\leq l$ (indices taken modulo $l$);
for each $j\leq q'$, there exists a unique $\overline{\rho}_{j}\in \mathrm{Aut}(\Gamma)$ such that $\overline{\rho}_{j}|_{W'}=\mathbf{id}$, where $W'$ is the
subspace spanned by $$\{z, u, x_{i}^{c},y_{i}^{c},\overline{x}_{i}^{c'},\overline{y}_{i}^{c'},a_{i}^{d},\overline{a}_{i}^{d'},p_{i'},q_{i'},z_{j'}\}$$ for $i\leq l$, $i'\leq s$, $j'\leq p$, $d\leq q$, $d'\leq q'$, $c\leq t$, $c'\leq t'$, $d'\neq j$, and  $$\overline{\rho}_{j}(\overline{a}_{i}^{j})=\overline{a}_{\frac{l}{2}+i}^{j}$$ for all $i\leq l$ (indices taken modulo $l$).

Finally take $\varrho_{i}$, $\widetilde{\tau}$ and $\widetilde{\varsigma}$ described in Theorem \ref{classification1}.

\begin{lemma}\label{W'}
Let $$\Gamma=\Gamma(l,t,t',q,q',s,p;\mathbf{m},\mathbf{m}',\mathbf{n},\mathbf{n}';\beta,\beta',\alpha,\alpha')$$ be a fine grading on $H_{n,m}^{\lambda,\kappa}$, where $$\mathbf{m}=(m_{1},\ldots,m_{\overline{t}}),~~\mathbf{m}'=(m_{1}',\ldots,m_{\overline{t}'}'),~~\mathbf{n}=(n_{1},\ldots,n_{\overline{q}}),~~ \mathbf{n}'=(n_{1}',\ldots,\overline{n}'_{\overline{q}'}).$$ If $l$ is even and the automorphism $f\in \mathrm{Aut}(\Gamma)$  such that $f(z)=z$ and $f(u)=u$, then the image $[f]$ of $f$ in $\mathcal{W}(\Gamma)$ belongs to the group generated by
\begin{equation}\label{Weyl group1}
\{[\Upsilon_{(\eta,\eta',\sigma,\sigma')}],[\theta_{j}],[\vartheta_{j}],[\overline{\theta}_{j'}],[\overline{\vartheta}_{j'}],[\rho_{i}],
[\overline{\rho}_{i'}][\varrho_{h}],[\widetilde{\tau}],[\widetilde{\varsigma}]\},\end{equation}
where $j\leq t,j'\leq t',i\leq q,i'\leq q',h\leq s$, $\eta\in S_{m_{1}}\times\cdots\times S_{m_{\overline{t}}}$, $\eta'\in S_{m_{1}'}\times\cdots\times S_{m_{\overline{t}'}'}$, $\sigma\in S_{n_{1}}\times\cdots\times S_{n_{\overline{q}}}$ and $\sigma'\in S_{n_{1}'}\times\cdots\times S_{n_{\overline{q}'}'}$.
\end{lemma}
\begin{proof}
Fix $j\in\{1,\ldots,t\}$. There is $c\in \{1,\ldots,\overline{t}\}$ such that $\delta_{c}=\beta_{j}$ (see Eq. (\ref{scalar chosen})). Suppose $X_{j}$ is the $j$th  block, which is a block of type $B_{l}^{\mathrm{I}}(u,\delta_{c})$. Since $f(u)=u$, we have $f(\langle X_{j}\rangle)=\langle X_{j}\rangle$, so it must generate the subspace spanned by the $d$th block for some $d$ such that $\delta_{c}=\beta_{d}$. In other words, there is $\eta\in S_{m_{1}}\times\cdots\times S_{m_{\overline{t}}}$ such that $f$ applies the $j$th block of type I into the span of the $\eta(j)=d$th block of type I. In the same way, there are $\eta'\in S_{m_{1}'}\times\cdots\times S_{m_{\overline{t}'}'}$, $\sigma\in S_{n_{1}}\times\cdots\times S_{n_{\overline{q}}}$, $\sigma'\in S_{n_{1}'}\times\cdots\times S_{n_{\overline{q}'}'}$ such that the automorphism $\Upsilon_{(\eta,\eta',\sigma,\sigma')}^{-1}f$ applies each block into the subspace spanned by itself. So we have $[f]$ is generated by $[\Upsilon_{(\eta,\eta',\sigma,\sigma')}]$, $[\theta_{j}]$, $[\vartheta_{j}]$, $[\overline{\theta}_{j'}]$, $[\overline{\vartheta}_{j'}]$, $[\rho_{i}]$,
$[\overline{\rho}_{i'}]$, $[\varrho_{h}]$, $[\widetilde{\tau}]$ and $[\widetilde{\varsigma}]$.
\end{proof}
Note that $\langle[\theta_{j}],[\vartheta_{j}]\rangle\cong D_{l}$, $\langle[\overline{\theta}_{j}],[\overline{\vartheta}_{j}]\rangle\cong D_{l}$, the group $\mathcal{W}'$ generated by the set in Eq. (\ref{Weyl group1}) is isomorphic to
\begin{eqnarray}
(S_{m_{1}}\times\cdots\times S_{m_{\overline{t}}}\times S_{m_{1}'}\times\cdots\times S_{m_{\overline{t}'}'}&\times S_{n_{1}}\times\cdots\times S_{n_{\overline{q}}}\times S_{n_{1}'}\times\cdots\times S_{n_{\overline{q}'}'}\times S_{s}\times S_{p})\nonumber\\
&\ltimes(D_{l}^{t+t'}\times \mathbb{Z}_{2}^{q+q'}\times \mathbb{Z}_{2}^{s}).\label{liu999}
\end{eqnarray}
Similar to \cite[Lemma 7]{Cal2}, we have the following  technical lemma.
\begin{lemma}\label{gc}
Let $$\Gamma=\Gamma(l,t,t',q,q',s,p;\mathbf{m},\mathbf{m}',\mathbf{n},\mathbf{n}';\beta,\beta',\alpha,\alpha')$$ be a fine grading on $L=H_{n,m}^{\lambda,\kappa}$, where $$\mathbf{m}=(m_{1},\ldots,m_{\overline{t}}),~~\mathbf{m}'=(m_{1}',\ldots,m_{\overline{t}'}'),~~\mathbf{n}=(n_{1},\ldots,n_{\overline{q}}),~~ \mathbf{n}'=(n_{1}',\ldots,\overline{n}'_{\overline{q}'}).$$
$$\beta=(\beta_{1},\ldots,\beta_{t}),~~ \beta'=(\beta_{1}',\ldots,\beta_{t'}'),~~ \alpha=(\alpha_{1},\ldots,\alpha_{q}),~~ \alpha'=(\alpha_{1}',\ldots,\alpha_{q'}').$$

  $(\mathrm{1})$ If $f\in \mathrm{Aut}(L)$ satisfies that $f(z)=z$ and $f(u)=\varepsilon u$, then $\varepsilon$ is a primitive $c$th root of  unit and

\begin{equation}\label{XX'}
 \{\widehat{\alpha_{1}},\ldots,\widehat{\alpha_{q}}\}=\bigcup\limits_{j=0}^{c-1}\varepsilon^{j}X,~~~\{\widehat{\alpha'_{1}},
\ldots,\widehat{\alpha'_{q'}}\}=\bigcup\limits_{j=0}^{c-1}\varepsilon^{j}X',
\end{equation}
\begin{equation}\label{YY'}
  \{\widehat{\beta_{1}},\ldots,\widehat{\beta_{t}}\}=\bigcup\limits_{j=0}^{c-1}\varepsilon^{j}Y,~~~\{\widehat{\beta'_{1}},
\ldots,\widehat{\beta'_{t'}}\}=\bigcup\limits_{j=0}^{c-1}\varepsilon^{j}Y',
\end{equation}
for some sets $X$, $X'$, $Y$ and $Y'$ of classes such that
$X\cap \varepsilon^{j} X$ is either $\emptyset$ or $X$, $X'\cap \varepsilon^{j} X'$ is either $\emptyset$ or $X'$, $Y\cap \varepsilon^{j} Y$ is either $\emptyset$ or $Y$ and $Y'\cap \varepsilon^{j} Y'$ is either $\emptyset$ or $Y'$ for all $0\leq j\leq c-1$.

  $(\mathrm{2})$ If there is a divisor $c$ of $t$, $t'$, $q$ and $q'$ such that $\beta_{j}=\beta_{j+\frac{t}{c}}$ for all $1\leq j\leq \frac{t}{c}$, $\beta'_{j}=\beta_{j+\frac{t'}{c}}'$ for all $1\leq j\leq \frac{t'}{c}$, $\alpha_{j}=\alpha_{j+\frac{q}{c}}$ for all $1\leq j\leq \frac{q}{c}$, $\alpha'_{j}=\alpha_{j+\frac{q'}{c}}'$ for all $1\leq j\leq \frac{q'}{c}$, for $\epsilon$ a primitive $c$th root of  unit, there exists a unique automorphism $g_{c}$ of order $c$ such that $g_{c}(u)=\frac{1}{\epsilon}u$, $g_{c}|_{V}=\mathbf{id}$, where $V$ is the subspace spanned by $\{z, p_{i}, q_{i}, z_{j}\mid i\leq s,j\leq p\}$,  and
  \begin{eqnarray*}
     &&  g_{c}(x_{i}^{j})=x_{i}^{j+\frac{t}{c}},~~~g_{c}(y_{i}^{j})=y_{i}^{j+\frac{t}{c}},~~~g_{c}(\overline{x}_{i}^{j'})=\overline{x}_{i}^{j'+\frac{t'}{c}},\\
     &&  g_{c}(\overline{y}_{i}^{j'})=\overline{y}_{i}^{j'+\frac{t'}{c}},~~~g_{c}(a_{i}^{d})=a_{i}^{d+\frac{q}{c}},~~~g_{c}(\overline{a}_{i}^{d'})= \overline{a}_{i}^{d'+\frac{q'}{c}}
  \end{eqnarray*}
 for all $i\leq l$, $j\leq t$, $j'\leq t'$, $d\leq q$, $d'\leq q'$.
\end{lemma}

Finally, we can
characterize  all the Weyl groups of for the fine gradings of twisted Heisenberg superalgebras.

\begin{theorem}
Let $$\Gamma=\Gamma(l,t,t',q,q',s,p;\mathbf{m},\mathbf{m}',\mathbf{n},\mathbf{n}';\beta,\beta',\alpha,\alpha')$$ be a fine grading on $H_{n,m}^{\lambda,\kappa}$, where $$\mathbf{m}=(m_{1},\ldots,m_{\overline{t}}),~~\mathbf{m}'=(m_{1}',\ldots,m_{\overline{t}'}'),~~\mathbf{n}=(n_{1},\ldots,n_{\overline{q}}),~~ \mathbf{n}'=(n_{1}',\ldots,\overline{n}'_{\overline{q}'}).$$
Take $c\in \mathbb{N}$ to be the greatest integer such that, for $\epsilon$ a primitive $c$th root of  unit, there are sets $X$, $X'$, $Y$ and $Y'$ of classes such that Eqs. (\ref{XX'}) and  (\ref{YY'}) holds. Let $d$ be the minimum positive integer such that $\widehat{\epsilon^{d}}=\widehat{1}$.

$(\mathrm{1})$ If $l$ is even, the Weyl group $\mathcal{W}(\Gamma)$ is isomorphic to the quotient group of
$\mathcal{W}'\rtimes \mathbb{Z}_{c}$ modulo $\mathbb{Z}_{\frac{c}{d}}$,
where
 $\mathcal{W}'$ is  the semidirect product of
$$S_{m_{1}}\times\cdots\times S_{m_{\overline{t}}}\times S_{m_{1}'}\times\cdots\times S_{m_{\overline{t}'}'}\times S_{n_{1}}\times\cdots\times S_{n_{\overline{q}}}\times S_{n_{1}'}\times\cdots\times S_{n_{\overline{q}'}'}\times S_{s}\times S_{p}$$
and
$D_{l}^{t+t'}\times \mathbb{Z}_{2}^{q+q'}\times \mathbb{Z}_{2}^{s}$.

$(\mathrm{2})$ If $l$ is odd, the Weyl group $\mathcal{W}(\Gamma)$ is isomorphic to the quotient group of $$((S_{m_{1}}\times\cdots\times S_{m_{\overline{t}}}\times S_{m_{1}'}\times\cdots\times S_{m_{\overline{t}'}'}\times \times \mathbb{Z}_{2}^{s+1}\times S_{s}\times S_{p})\ltimes \mathbb{Z}_{l}^{t+t'})\rtimes \mathbb{Z}_{c}$$
modulo $\mathbb{Z}_{\frac{c}{d}}$.
\end{theorem}
\begin{proof}
 Suppose the grading $\Gamma$ is formed by $z$, $u$, blocks $\{x_{1}^{j},y_{1}^{j},\ldots,x_{l}^{j},y_{l}^{j}\}$ of type $B_{l}^{\mathrm{I}}(u,\beta_{j})$ for $j\leq t$, blocks $\{\overline{x}_{1}^{j},\overline{y}_{1}^{j},\ldots,\overline{x}_{l}^{j},\overline{y}_{l}^{j}\}$ of type $\overline{B}_{l}^{\mathrm{I}}(u,\beta_{j}')$ for $j\leq t'$, blocks
$\{a_{1}^{j},\ldots,a_{l}^{j}\}$ of type $B_{\frac{l}{2}}^{\mathrm{II}}(u,\alpha_{j})$ for $j\leq q$, blocks $\{\overline{a}_{1}^{j},\ldots,\overline{a}_{l}^{j}\}$ of type $\overline{B}_{\frac{l}{2}}^{\mathrm{II}}(u,\alpha_{j}')$ for $j\leq q'$, and the block $\{p_{1},q_{1},\ldots,p_{s},q_{s},z_{1},\ldots,z_{p}\}$ of type $B_{s,p}^{\mathrm{III}}(u)$.

(1)  Suppose $l$ is even. Take $f\in \mathrm{Aut}(\Gamma)$. Then $f(z)\in \langle z\rangle$ and $f(u)\in \langle u\rangle$. We can assume that $f(z)=z$, by replacing $f$ with the composition of $f$ with the automorphism given by $$z\mapsto \alpha^{2}z,~~u\mapsto u,~~x_{i}^{j}\mapsto \alpha x_{i}^{j},~~y_{i}^{j}\mapsto \alpha y_{i}^{j},~~
\overline{x}_{i}^{j}\mapsto \alpha \overline{x}_{i}^{j},~~\overline{y}_{i}^{j}\mapsto \alpha \overline{y}_{i}^{j},$$ $$ a_{i}^{j}\mapsto \alpha a_{i}^{j},~~~ \overline{a}_{i}^{j}\mapsto \alpha \overline{a}_{i}^{j},~~~p_{i}\mapsto \alpha p_{i},~~~ q_{i}\mapsto \alpha q_{i},~~~ z_{i}\mapsto \alpha z_{i}$$
for some $\alpha\in \mathbb{C}^{\times}$.

Consider $g_{c}$ the automorphism described in Lemma \ref{gc} (2). Note that if we apply Lemma \ref{gc} (1) to $f$, there is $\varepsilon$ a primitive $c'$th root of  unit such that $f(u)=\varepsilon u.$ So $\varepsilon\epsilon$ is a primitive $\mathrm{lcm}(c,c')$ root of unit and $fg_{c}^{-1}(u)=\varepsilon\epsilon u$.
By Lemma \ref{gc} (1), we have the corresponding splitting. By maximality of $c$, we conclude that $\mathrm{lcm}(c,c')=c$ and $\varepsilon$ is a power of $\epsilon$. Thus there is $h\in \mathbb{N}$ such that $fg_{c}^{h}(u)=u$ and thus Lemma \ref{W'} is applied and $[f]\in \langle [g_{c}], \mathcal{W}'\rangle$, where $\mathcal{W}'$ is described as in Eq. (\ref{liu999}).
Moreover, $\mathcal{W}(\Gamma)=\mathcal{W}'\rtimes\langle[g_{c}]\rangle$, since $g_{c}fg_{c}^{-1}=\varepsilon u$ for any $f\in \mathcal{W}'$.

By Lemma \ref{gc} (2), the new generator $[g_{c}]$ has order $c$. Since $g_{c}^{d}=\epsilon^{q}u,$ there exists some element $f'\in \mathrm{Stab}(\Gamma)$ such that $f'g_{c}^{d}(u)=u$ and $f'g_{c}^{d}(z)=z.$ Thus $[g_{c}]^{d}\in \mathcal{W}'$. So we have $\mathcal{W}(\Gamma)$ is isomorphic to the quotient group of  $\mathcal{W}'\rtimes \mathbb{Z}_{c}$ modulo $\mathbb{Z}_{\frac{c}{d}}$.

(2) Suppose $l$ is odd.
Now the maps $\vartheta_{j}$ and $\overline{\vartheta_{j}}$ are not longer automorphisms, but there exists $\vartheta'\in \mathrm{Aut}(\Gamma)$ such that  $\vartheta'(u)=-u,$ $\vartheta'|_{V}=\mathbf{id}$, where $V$ is the subspace spanned by $$\{z,p_{i},q_{i},z_{j}\mid i\leq s,j\leq p\},$$
and $$\vartheta'(x_{i}^{j})=(-1)^{i}y_{i}^{j},~~~\vartheta'(y_{i}^{j})=(-1)^{i+1}x_{i}^{j},~~~\vartheta'(\overline{x}_{i}^{j})=(-1)^{i}
\overline{y}_{i}^{j},~~~\vartheta'(\overline{y}_{i}^{j})=(-1)^{i}
\overline{x}_{i}^{j}.$$
If $f\in \mathrm{Aut}(\Gamma)$ fixes the subspaces spanned by the blocks, it is not difficult to check that $[f]$ belongs to the subgroup generated by
$$\{[\vartheta'],[\theta_{j}],[\overline{\theta}_{j'}],[\varrho_{i}],[\widetilde{\tau}],[\widetilde{\varsigma}]:j\leq t,j'\leq t',i\leq s\},$$
which is isomorphic to $\mathbb{Z}_{l}^{t+t'}\rtimes (\mathbb{Z}_{2}^{s+1}\times S_{s}\times S_{p})$. Take $g_{c}$ as in the even case. Similar to the even case, we have $\mathcal{W}(\Gamma)=\mathcal{W}'\rtimes\langle[g_{c}]\rangle$,
where $$\mathcal{W}'=(S_{m_{1}}\times\cdots\times S_{m_{\overline{t}}}\times S_{m_{1}'}\times\cdots\times S_{m_{\overline{t}'}'}\times  \mathbb{Z}_{2}^{s+1}\times S_{s}\times S_{p})\ltimes \mathbb{Z}_{l}^{t+t'}.$$
Hence  $\mathcal{W}(\Gamma)$ is isomorphic to
$$\frac{((S_{m_{1}}\times\cdots\times S_{m_{\overline{t}}}\times S_{m_{1}'}\times\cdots\times S_{m_{\overline{t}'}'}\times \mathbb{Z}_{2}^{s+1}\times S_{s}\times S_{p})\ltimes \mathbb{Z}_{l}^{t+t'})\rtimes \mathbb{Z}_{c}}{\mathbb{Z}_{\frac{c}{d}}}.$$
\end{proof}

\end{document}